\documentclass[10pt,reqno]{amsart}
\include{formatAndDefs}
\usepackage{amsfonts,amssymb,amsmath}

\parskip        \baselineskip
\usepackage{aliascnt}
\usepackage[mathscr]{eucal}
\usepackage{graphicx}
\usepackage{latexsym}
\usepackage{psfrag}
\usepackage{epsfig}
\usepackage[utf8]{inputenc}
\usepackage[english]{babel}
\usepackage{tikz}
\usepackage[foot]{amsaddr}
\usepackage{lineno, blindtext}
\usetikzlibrary{decorations.pathreplacing}
\newcommand{\suchthat}{\;\ifnum\currentgrouptype=16 \middle\fi|\;}

\renewcommand{\leq}{\leqslant}
\renewcommand{\geq}{\geqslant}

\newcommand{\ds}{\displaystyle}
\numberwithin{equation}{section}
\newtheorem{thm}{Theorem}
\numberwithin{thm}{section}

\newaliascnt{lemma}{thm}
\newtheorem{lem}[lemma]{Lemma}
\aliascntresetthe{lemma}

\newaliascnt{proposition}{thm}

\aliascntresetthe{proposition}

\newaliascnt{corollary}{thm}
\newtheorem{corollary}[corollary]{Corollary}
\aliascntresetthe{corollary}

\newaliascnt{definition}{thm}

\aliascntresetthe{definition}

\newaliascnt{remark}{thm}
\newtheorem{remark}[remark]{Remark}
\aliascntresetthe{remark}

\usepackage[colorlinks=true]{hyperref}

\begin{document}

\title[Carleman estimate for an adjoint of a damped beam equation and an application]{Carleman estimate for an adjoint of a damped beam equation and an application to null controllability}

\date{\today}
\author{Sourav Mitra}
\thanks{{Acknowledgments}: The work was done as a part of the PHD thesis of the author in Institut de Math\'ematiques de Toulouse. The author wishes to thank the ANR project ANR-15-CE40-0010 IFSMACS as well as the Indo-French Centre for Applied Mathematics (IFCAM) for the funding provided during this work.}
\address{Sourav Mitra, Institute of Mathematics, University of Würzburg, 97074, Germany}
\email{sourav.mitra@mathematik.uni-wuerzburg.de,\,Tél: +49 931 31-89531,\,Fax: +49 931 31-80944}

\begin{abstract}
In this article we consider a control problem of a linear Euler-Bernoulli damped beam equation with potential in dimension one with periodic boundary conditions. We derive a new Carleman estimate for an adjoint of the equation under consideration. Then using a well known duality argument we obtain explicitly the control function which can be used to drive the solution trajectory of the control problem to zero state. 
\end{abstract}

\maketitle

\section{Introduction}
\subsection{Statement of the problem}
Let $d$ and $L$ are positive constants and ${\mathbb{T}^{d}_{L}}$ is the one dimensional torus identified with $(-L,d+L)$ with periodic conditions. In this article we consider the following control problem corresponding to the Euler-Bernoulli damped beam equation:
\begin{equation}\label{dampedbeam}
\left\{ \begin{array}{ll}
\displaystyle\partial_{tt}\beta-\partial_{txx}\beta+\partial_{xxxx}\beta+a\beta=v_{\beta}\chi_{\omega} \,& \mbox{in}\, {\mathbb{T}^{d}_{L}}\times(0,T),
\vspace{1.mm}\\
\displaystyle\beta(\cdot,0)=\beta_{0}\quad \mbox{and}\quad \partial_{t}\beta(\cdot,0)=\beta_{1}\, &\mbox{in}\, {\mathbb{T}^{d}_{L}},
\end{array}\right.
\end{equation}
  where $a=a(x,t)\in L^{\infty}(\mathbb{T}^{d}_{L}\times(0,T))$ is a potential,
  \begin{equation}\label{controlzone}
  \begin{array}{l}
  \omega=(-L,0)\cup(d,d+L)
  \end{array}
  \end{equation}
  and
 $\chi_{\omega}$ represents the characteristics function corresponding to the set $\omega.$ The set $\omega$ will correspond to the boundary control zone for the damped beam equation.\\ 
 The central theorem of the present article is based on proving a new Carleman estimate for the formal adjoint to the differential operator 
 \begin{equation}\label{formaloperator}
 \begin{array}{l}
 (\partial_{tt}-\partial_{txx}+\partial_{xxxx}),
 \end{array}
 \end{equation}
 with an observation on the set $\omega\times(0,T).$ As an application of this Carleman estimate we then construct a suitable boundary control function $v_{\beta}\chi_{\omega}$ such that the unknowns $\beta,$ the beam displacement and $\partial_{t}\beta,$ the beam velocity satisfy the following controllability requirement
 \begin{equation}\label{controlrequirement}
 \begin{array}{l}
 (\beta,\partial_{t}\beta)(\cdot,T)=0,
 \end{array}
 \end{equation}
 for some positive time $T>0.$ There is no restriction over the controllability time $T.$\\
 To state the central result of the present article we have to introduce some suitable weight functions.
 \subsection{Construction of the weight functions}\label{Conw8fn}
 Let $s (\geqslant 1)$ and $\lambda(\geqslant 1)$ be two positive parameters.\\ 
 $1.$ We first introduce a function $\eta$ on ${\mathbb{T}^{d}_{L}}$ such that 
 \begin{equation}\label{eta*}
 \begin{array}{l}
 \eta\in C^{6}({\mathbb{T}}_{L}),\,\, \eta(x)>0\,\,\mbox{in}\,\, {\mathbb{T}^{d}_{L}},
 \\
 \mbox{inf} \left\{|\nabla\eta(x)|
 \suchthat x\in{\mathbb{T}^{d}_{L}}
 \setminus\omega\right\}
 >0.
 \end{array}
 \end{equation}
 %Now we define $\eta^{0}\in C^{6}({\mathbb{T}^{d}_{L}}\times[0,T])$ as follows
 %\begin{equation}\label{defeta0}
 %\begin{array}{l}
 %\eta^{0}(x,t)=\eta(x-\overline{u}_{1}t)\quad\mbox{for all}\quad (x,t)\in {\mathbb{T}^{d}_{L}}\times[0,T] .
 %\end{array}
 %\end{equation}
 %In view of \eqref{eta*}, one can easily verify the following
 %\begin{equation}\label{eta0grtr0}
 %\begin{array}{l}
 %\mbox{inf}\left\{|\nabla\eta^{0}(x,t)|\suchthat (x,t)\in [-\overline{u}_1 T, d + \overline u_1 T]\times[0,T]\right\}>0,
 %\end{array}
 %\end{equation}
 %
 %Since in the definition \eqref{defeta0}, $\eta^{0}$ only depends on $x$ and $t$ variables, from now on we will simply use the following notation
 %\begin{equation}\label{sup1}
 %\eta^{0}(x,t)=\eta^{0}(x,t). 
 %\end{equation}
 $2.$ Next we will define a weight function in the time variable. Let  $T_{0}>0,$ $T_{1}>0,$ are such that
 \begin{equation}\label{relT01ep}
 \begin{split}
 2T_{0}+2T_{1}<T.
 \end{split}
 \end{equation}
 Now we choose a weight function $\theta(t) \in C^4(0,T)$ such that
 \begin{equation}\label{theta}
 \theta(t)=\left\{ \begin{array}{lll}
 \displaystyle
 &\displaystyle \frac{1}{t^{2}},\,\,&\forall\,\, t\in[0,T_{0}],\smallskip\\
 &\theta\, \mbox{is strictly decreasing}\,\,&\forall\,\, t\in[T_{0}, 2T_0],\\
 & 1\,\,&\forall\,\, t\in[2 T_{0},T-2T_{1}],\\
 & \theta\, \mbox{is strictly increasing}\,\,&\forall\,\, t\in[T-2T_{1},T-T_{1}],\\
 & \displaystyle \frac{1}{(T-t)^{2}},\,\,&\forall\,\, t\in[T-T_{1},T].
 \end{array}\right.
 \end{equation}
 Observe that $\theta(t)$ blows up at the terminal points $\{0\}$ and $\{T\}$ of the interval $(0,T).$\\
 $3.$ In view of $\eta$ and $\theta(t)$ we finally introduce the following weight functions in ${\mathbb{T}^{d}_{L}}\times[0,T],$
 \begin{equation}\label{w8fn}
 \left\{ \begin{array}{l}
 \displaystyle\phi(x,t)=\theta(t)(e^{6\lambda\|\eta\|_{\infty}}-e^{\lambda(\eta+4\|\eta\|_{\infty})}),\\
 {\xi}(x,t)=\theta(t)e^{\lambda(\eta+4\|\eta\|_{\infty})}.
 \end{array}\right.
 \end{equation}
 %From now on we will use the following notations
 %\begin{equation}\label{phixi}
 %%\phi(x,t)=\phi(x,1,t)\quad\mbox{and}\quad \xi(x,t)=\xi(x,1,t).
 %\end{array}
 %\end{equation}
 From now on until the end of this article we will denote by $c,$ a generic strictly positive small constant and by $C,$ a large constant, where both of them are independent of the parameters $s$ ($\geq 1$) and $\lambda$ ($\geq 1$).\\
 Note that the weight functions defined above closely relates with the weight functions used in proving Carleman estimate for adjoint heat equation. The choice that $\theta(t)$ equals one in a subinterval of $(0,T)$ is done to apply the Carleman estimate in studying the controllability of coupled PDE problems (especially parabolic hyperbolic coupling) in further works. For similar issues of controllability of coupled parabolic hyperbolic system one can consult the articles \cite{ervgugla} and \cite{ervbad}. One can also look into \cite[Chapter 4]{phdthesis} for the application of the Carleman estimate proved in this article to study the observability property of  a compressible fluid structure interaction problem. 
 \\
  Now let us state the following result corresponding to the Carleman estimate of the formal adjoint to the operator \eqref{formaloperator}.
 \begin{thm}\label{Carlbeamthm}
 There exist a constants  $C>0,$ $s_{0}\geqslant 1$ and $\lambda_{0}\geqslant 1$ such that for all smooth functions $\psi$ on ${\mathbb{T}^{d}_{L}}\times[0,T],$ for all $s\geqslant s_{0}$ and $\lambda\geqslant\lambda_{0},$
 	\begin{align}\label{crlestbmthm}
 	& \displaystyle s^{7}\lambda^{8}\int_0^T \int_{{\mathbb{T}^{d}_{L}}} {\xi}^{7}|{\psi}|^{2}e^{-2s\phi}+s^{5}\lambda^{6}\int_0^T \int_{{\mathbb{T}^{d}_{L}}} {\xi}^{5}|{\partial_{x}{\psi}}|^{2}e^{-2s\phi}\notag
 	%	\\
 	% 	&\displaystyle
 	\\
 	&\displaystyle+s^{3}\lambda^{4}\int_0^T \int_{{\mathbb{T}^{d}_{L}}} {\xi}^{3} ( |{\partial_{xx}{\psi}}|^{2}+ |{\partial_{t}{\psi}}|^{2}) e^{-2s\phi}+s\lambda^{2}\int_0^T \int_{{\mathbb{T}^{d}_{L}}} {\xi}( |{\partial_{tx}{\psi}}|^{2}+|{\partial_{xxx}{\psi}}|^{2})e^{-2s\phi}\notag \\
 	%	\\
 	%	&
 	&+
 	\displaystyle\frac{1}{s}\int_0^T \int_{{\mathbb{T}^{d}_{L}}} \frac{1}{{\xi}}(|{\partial_{tt}{\psi}}|^{2}+|\partial_{txx}\psi|^{2}+|{\partial_{xxxx}{\psi}}|^{2})e^{-2s\phi}
 	\\
 	&
 	\displaystyle\leqslant  
 	C\int_0^T \int_{{\mathbb{T}^{d}_{L}}} |(\partial_{tt}+\partial_{txx}+\partial_{xxxx})\psi|^{2}e^{-2s\phi}
 	+Cs^{7}\lambda^{8}\int_0^T \int_{\omega} {\xi}^{7}|{\psi}|^{2}e^{-2s\phi},\notag
 	\end{align}
 	where the notation $\omega$ was introduced in \eqref{controlzone}.
 \end{thm}
 The next theorem corresponds to a null controllability result for a damped beam equation with potential which is derived as an application of the Theorem \ref{Carlbeamthm}. We will use in particular a Corollary \ref{corpotential} of Theorem \ref{Carlbeamthm} to prove the following result.
 \begin{thm}\label{Centraltheoremnullcont}
 	Let $T>0,$ $a\in L^{\infty}(\mathbb{T}^{d}_{L}\times(0,T))$ be a potential and the initial datum satisfy the following regularity assumptions:
 	\begin{equation}\label{regularityinitial}
 	\begin{array}{l}
 	\beta_{0}\in H^{3}({\mathbb{T}^{d}_{L}}),\quad\mbox{and}\quad \beta_{1}\in H^{1}({\mathbb{T}^{d}_{L}}).
 	\end{array}
 	\end{equation}
 	There exists a control $v_{\beta}\in L^{2}(0,T;L^{2}({\mathbb{T}^{d}_{L}})),$ such that the solution to the system \eqref{dampedbeam}
 	 satisfies the null controllability requirement \eqref{controlrequirement} and the controlled trajectory $\beta$ has the following regularity
 	\begin{equation}\label{regularitytrajectory}
 	\begin{array}{l}
 	\beta\in L^{2}(0,T;H^{4}({\mathbb{T}^{d}_{L}}))\cap H^{1}(0,T;H^{2}({\mathbb{T}^{d}_{L}}))\cap H^{2}(0,T;L^{2}({\mathbb{T}^{d}_{L}})).
 	\end{array}
 	\end{equation}
 \end{thm} 
 The equation \eqref{dampedbeam}$_{1}$ we consider is of parabolic nature. In other words the operator 
 \begin{equation}\label{operatorA}
 \begin{array}{l}
 \mathcal{A}=\begin{pmatrix}
 0 & I\\
 -\partial_{xxxx} & \partial_{xx}
 \end{pmatrix}
 \end{array}
 \end{equation}
(the operator $\mathcal{A}$ is without potential) defined in $H^{2}({\mathbb{T}^{d}_{L}})\times L^{2}({\mathbb{T}^{d}_{L}})$ with the domain
 $$\mathcal{D}(\mathcal{A})=H^{4}({\mathbb{T}^{d}_{L}})\times H^{2}({\mathbb{T}^{d}_{L}}),$$ is the generator of an analytic semigroup. For details we refer the readers to \cite{chen}. The well posedness of the system \eqref{dampedbeam} with $a=0$ is well studied in the literature and we will comment more on that afterwards. In our case since the system \eqref{dampedbeam} is with potential, we state the following result for the well posedness and regularity of system \eqref{dampedbeam}.
	\begin{lem}\label{existencepotential}
		Let, $a=a(x,t)\in L^{\infty}(\mathbb{T}^{d}_{L}\times(0,{T}))$ be a potential. Let $$(\beta_{0},\beta_{1})\in H^{3}({\mathbb{T}^{d}_{L}})\times H^{1}({\mathbb{T}^{d}_{L}}),$$
		and the control function, $v_{\beta}\in L^{2}(0,T;\mathbb{T}^{d}_{L})$ satisfies
		\begin{equation}\label{boundvbeta}
		\begin{array}{l}
		\|v_{{\beta}}\|_{L^{2}(0,T;L^{2}({\mathbb{T}^{d}_{L}}))}\leqslant C\|(\beta_{0},\beta_{1})\|_{H^{3}({\mathbb{T}^{d}_{L}})\times H^{1}({\mathbb{T}^{d}_{L}})},
		\end{array}
		\end{equation}
		for some positive constant $C.$ Then the system \eqref{dampedbeam}
		%\begin{equation}\label{dampedbeamgeneral2}
		%\left\{ \begin{array}{ll}
		%\displaystyle\partial_{tt}\beta-\partial_{txx}\beta+\partial_{xxxx}\beta+a\beta=0 \,& \mbox{in}\, {\mathbb{T}^{d}_{L}}\times(0,{\kappa}),
		%\vspace{1.mm}\\
		%\displaystyle\beta(\cdot,0)=\beta_{0}\quad \mbox{and}\quad \partial_{t}\beta(\cdot,0)=\beta_{1}\, &\mbox{in}\, {\mathbb{T}^{d}_{L}},
		%\end{array}\right.
	    %\end{equation}
		admits a unique solution in the functional framework \eqref{regularitytrajectory}.
		Besides, there exists a positive constant $C,$  such that the following holds
		\begin{equation}\label{inequalityexistence2}
		\begin{array}{l}
		\displaystyle\|(\beta,\partial_{t}\beta)\|_{L^{2}(0,{T};H^{4}({\mathbb{T}^{d}_{L}})\times H^{2}({\mathbb{T}^{d}_{L}}))\cap H^{1}(0,{T};H^{2}({\mathbb{T}^{d}_{L}})\times L^{2}({\mathbb{T}^{d}_{L}}))}\leqslant\displaystyle C\|(\beta_{0},\beta_{1})\|_{H^{3}({\mathbb{T}^{d}_{L}})\times H^{1}({\mathbb{T}^{d}_{L}})}.
		\end{array}
		\end{equation}
	\end{lem}
We will recall the proof of Lemma \ref{existencepotential} in Section \ref{appendix}.\\
To the best of our knowledge our article is the first one proving a Carleman estimate for the adjoint of the damped beam equation \eqref{dampedbeam}. The null controllability problem \eqref{dampedbeam} without the potential term is already studied in the articles \cite{lasieckadampedbeam}, \cite{millerbeam} and \cite{julienedward} using spectral methods. Their technique is completely different from ours which is based in proving a Carleman estimate for the adjoint to the system \eqref{dampedbeam}.
 In \cite{lasieckadampedbeam} the authors consider a more general controllability problem:
 $$\displaystyle w_{tt} + Sw + \rho S^{\alpha}w_{t} = u;\quad w(0) = w_0;\quad w_t (0) = w_1;\quad \rho> 0;\quad 1/2 \leqslant \alpha \leqslant 1,$$
 where for some Hilbert space $\mathcal{X},$ $S:\mathcal{D}(S)(\subset \mathcal{X})\longrightarrow \mathcal{X}$ is a positive, self-adjoint, unbounded operator with compact resolvent. The control $u$ is not localized and is assumed to be distributed over the whole domain. In \cite{julienedward} the authors consider a one dimensional damped beam (similar to the one \eqref{dampedbeam}$_{1}$ but without the potential) with hinged ends and with a positive parameter $\rho$ appearing as the coefficient of $\partial_{txx}\beta.$ They study the null controllability of the system with a localized interior control by proving an observability inequality uniformly with respect to $\rho.$ The approach of both the articles \cite{lasieckadampedbeam} and \cite{julienedward} is based on proving an observability estimate by using Fourier decomposition and suitably using Bassel's inequality and Ingham-type inequalities for complex frequencies. In \cite{millerbeam} the author explicitly obtains the cost of the control as $T\longrightarrow 0,$ by tracking the constants in the observability estimate using spectral methods.\\
 The main focus of the present article is to derive a new Carleman estimate for the dual to the problem \eqref{dampedbeam}. Then using a duality argument we prove the null controllability of the primal problem \eqref{dampedbeam}. The duality argument used in this article is motivated from \cite{farnan} and \cite{ervbad}. In fact the concept of duality between controllability and observability dates back to the celebrated Hilbert Uniqueness method (HUM), introduced in the article \cite{JLLionscontexact}, which reduces the question of exact controllability problem of a partial differential equation into proving the observability estimate of the corresponding adjoint problem.\\
 The Carleman estimate obtained in this article can be used to prove controllability results corresponding to more complicated coupled dynamical systems, like the ones considered in \cite{raymondbeam} and \cite{Mitraexistence}.  In fact the two main advantages of using Carleman estimate in studying the controllability properties of a PDE are:
 $(i)\,$ Using suitably large Carleman parameter $s$ one can readily incorporate lower order terms especially a $L^{\infty}$ potential in a linear PDE model to study the controllability. Whereas spectral methods can not be applied in analyzing the controllability properties of a linear PDE model with potential (or with other lower order terms). This specific advantage of Carleman estimate is often exploited to deal with the controllability issues of semi linear PDE models.\\
 $(ii)\,$ Moreover, to track the behavior of the  spectrum of coupled PDE models is often very complicated. For example of such models one can have a look in the fluid structure interaction problems (with an elastic structure at the boundary) considered in \cite[Chapter 4]{phdthesis} and \cite{raymondbeam}. Carleman estimate can prove to be a very useful tool for studying controllability issues of such coupled problems. In connection with this discussion we would like to refer to \cite[Chapter 4]{phdthesis} where the author obtains an observability inequality for a compressible fluid structure interaction problem using Carleman estimates for some decoupled equations.\\ [2.mm]
 Let us briefly discuss the strategy of the present article in the following.
 \subsection{Comments on the Strategy} 
 (i)\,\,$\mathit{The\,\, Carleman\,\, estimate}$: The proof of the controllability result Theorem \ref{Centraltheoremnullcont} relies on studying the observability of the corresponding adjoint system. This observability is the consequence of the Carleman estimate stated in Theorem \ref{Carlbeamthm}, more specifically Corollary \ref{corpotential}.  In fact we prove a Carleman estimate for all smooth functions defined in ${\mathbb{T}^{d}_{L}}\times[0,T]$ with an observation in the set $\omega\times(0,T).$ Roughly speaking the Carleman estimate is a way to bound the weighted energy of a PDE system by just using the energy localized on the observation set $\omega\times(0,T).$ Thanks to the parabolic nature of the equation \eqref{dampedbeam}, we are able to prove a Carleman estimate by using similar weight functions which are used in the literature in deriving Carleman estimate for heat equation (for instance one can look into the articles \cite{fursikov}, \cite{farnan} and \cite{ervbad}). Unlike the heat equation in our case the damped beam equation consists of second order derivative in time and fourth order derivative in space and this makes the proof of the Carleman estimate very involved and tricky. The weight function $\phi(x,t)$ we use roughly equals to $\theta(t)e^{6\lambda\|\eta\|_{\infty}},$ where the weight $\theta(t)$ in time blows up at terminal points $\{0\}$ and $\{T\},$ $\lambda$ is a positive parameter and $\eta(x)$ is sufficiently smooth positive valued function defined on ${\mathbb{T}^{d}_{L}}$ with all its critical points in the set $\omega.$\\
 Now to derive a Carleman estimate solved by a smooth function $\psi,$ the trick is to perform a change of unknown $w=e^{-s\phi}\psi,$ and introduce a new quantity
 $$P_{\phi}w=e^{-s\phi}({\partial_{tt}{\psi}}+{\partial_{txx}{\psi}}+{\partial_{xxxx}{\psi}})=e^{-s\phi}(\partial_{tt}(e^{s\phi}w)+\partial_{txx}(e^{s\phi}w)+\partial_{xxxx}(e^{s\phi}w)).$$ 
 Next the most important part of the analysis is to suitably decompose $P_{\phi}w$ as
 $$P_{\phi}w=P_{1}w+P_{2}w+\mathit{R}w,$$
 where $P_{1}$ and $P_{2}$ roughly corresponds to the formally computed symmetric and anti symmetric part of the operator $P_{\phi},$ whereas $\mathit{R}$ corresponds to the lower order terms. We have managed to incorporate a lower order term in the expression of $P_{2}w$ and show that the product term $\displaystyle\int_0^T \int_{{\mathbb{T}^{d}_{L}}} P_{1}wP_{2}w$ admits of positive coefficients except possibly on the observation set $\omega\times(0,T).$ This in turn is used to prove the claimed Carleman estimate. Consequently one can easily obtain a Carleman estimate of the adjoint of a damped beam equation with potential. For details we refer the readers to Corollary \ref{corpotential}. \\
 An alternative way to obtain a Carleman estimate corresponding to the operator $(\partial_{tt}+\partial_{txx}+\partial_{xxxx})$ (which is without the potential term) is to factorize the adjoint operator as follows:
 $$(\partial_{tt}+\partial_{txx}+\partial_{xxxx})\psi=(\partial_{t}+\frac{1\pm \sqrt{3}i}{2}\partial_{xx})(\partial_{t}+\frac{1\mp \sqrt{3}i}{2}\partial_{xx})\psi,$$
 and then use the Carleman estimate for parabolic equations with complex coefficients, for instance one can use the result form \cite{Xiaoyufu}. But in that case we can only obtain bound over $s^{6}\lambda^{8}\displaystyle\int_0^T \int_{{\mathbb{T}^{d}_{L}}}\xi^{6}|\psi|^{2}e^{-2s\phi},$ where $s$ and $\lambda$ are Carleman parameters and $\xi=\theta(t)e^{\lambda(\eta(x,t)+4\|\eta\|_{\infty})},$ but this result is not optimal. On the other hand the Carleman estimate stated in Theorem \ref{Carlbeamthm} derives a bound on $s^{7}\lambda^{8}\displaystyle\int_0^T \int_{{\mathbb{T}^{d}_{L}}}\xi^{7}|\psi|^{2}e^{-2s\phi},$ which seems to be optimal in the sense that the exponents of the parameters $s$ and $\lambda$ can not be improved. The optimality of exponents of Carleman parameters can play a crucial role while dealing with coupled PDE systems with strong coupling. One can for instance look into \cite[Chapter 4]{phdthesis} where a bound over $s^{6}\lambda^{8}\displaystyle\int_0^T \int_{{\mathbb{T}^{d}_{L}}}\xi^{6}|\psi|^{2}e^{-2s\phi},$ is not enough and one needs a bound over $s^{7}\lambda^{8}\displaystyle\int_0^T \int_{{\mathbb{T}^{d}_{L}}}\xi^{6}|\psi|^{2}e^{-2s\phi},$ to prove the observability properties of a compressible fluid structure interaction problem where a damped beam of the form \eqref{dampedbeam} appears at the fluid boundary. \\
 (ii)\,\,$\mathit{Null-controllability\,\,of\,\,\eqref{dampedbeam}}$: Next in Section \ref{Nullcontrollability} we prove Theorem \ref{Centraltheoremnullcont} by a duality argument. In fact we introduce a cutoff function in time and using this we reduce the control problem \eqref{dampedbeam}-\eqref{controlrequirement} into a homogeneous initial value null controllability problem. To prove the null controllability of the new problem we write it in a weak form and introduce a functional whose Euler-Lagrange equation coincides with the obtained weak formulation. This strategy is inspired from \cite{fursikov}, \cite{farnan} and \cite{ervbad} where the authors treat the null controllability problem of heat type equations using this technique. Then thanks to the Carleman estimates derived in Section \ref{Carlemanestimate}, we show that the functional admits of a unique minimizer in a suitable Hilbert space. This minimizer is eventually used to obtain an explicit expression of a control function and an expression of the controlled trajectory. We further obtain an estimate on the $L^{2}({\mathbb{T}^{d}_{L}}\times(0,T))$ norm of the control function which is eventually used to show that the controlled trajectory satisfies the regularity \eqref{regularitytrajectory} as a consequence of Lemma \ref{existencepotential}.\\
 Since we are considering a one dimensional beam with periodic boundary conditions one may use spectral methods to prove the null controllability of the system \eqref{dampedbeam} when the potential $a=0$. For instance taking the Fourier transform of \eqref{dampedbeam}$_{1}$ with potential $a=0$ it is not hard to compute the following expression of eigenvalues and eigenfunctions corresponding to the operator $\mathcal{A}$ (given by \eqref{operatorA}):
 \begin{equation}\label{eigenvaluevector}
 \begin{array}{l}
 \displaystyle\lambda_{k}=\frac{-k^{2}\pm\sqrt{3}ik^{2}}{2},\quad \delta_{k}=\begin{pmatrix}
 e^{ikx}\\
 \lambda_{k}e^{ikx}
 \end{pmatrix},\quad\mbox{for\,\,all}\quad k\in \mathbb{Z}.
 \end{array}
 \end{equation}   
 It can be checked that $\mathcal{A}^{*},$ the adjoint of $\mathcal{A}$ with $a=0$ computed in the inner product of $L^{2}({\mathbb{T}^{d}_{L}})\times L^{2}({\mathbb{T}^{d}_{L}}),$ admits of same eigenvalues and eigenvectors as of $\mathcal{A}$ with $a=0$ given by \eqref{eigenvaluevector}. Next one can exploit the gap property in between two consecutive eigenvalues in order to apply spectral methods to prove null controllability of the system \eqref{dampedbeam} with $a=0$. For details we refer to the articles \cite{julienedward}, \cite{lasieckadampedbeam} and \cite{millerbeam}. In the present article we will further not discuss about the spectral methods and will rely on a new Carleman estimate which is the base of our analysis. Due to the strength of the Carleman parameters it is possible to handle the null controllability of a damped beam equation with a non trivial potential. Generalizing the Carleman estimate obtained in this article to dimension greater than one and to more general damped beam with general Lam\'{e} coefficients is a matter of future research. Nevertheless due to its plethora of applications (unique continuation, inverse problems etc.), Carleman estimate has its own interest.  
 \subsection{Bibliographical Comments} The well posedness of the system \eqref{dampedbeam} with $a=0$ in the framework of Hilbert space is well studied in the literature. In our case we have used the fact that the operator associated with \eqref{dampedbeam}$_{1}$ and $a=0$ is the generator of an analytic semigroup, for instance one can see Lemma \ref{lemmaexistence}. The result corresponding to the analyticity of the associated semigroup follows from \cite{chen} and \cite{chen2}. We further used this result to obtain a existence and regularity result for a damped beam equation with potential in Lemma \ref{existencepotential}. Maximal regularity in the $L^{p}-L^{q}$ regularity framework for a structurally damped beam with inhomogeneous Dirichlet-Neumann boundary condition is studied in the article \cite{denk}. The approach of \cite{denk} is mainly based on $\mathcal{R}-$ boundedness and Fourier multiplier theorems. An unified approach to the existence, uniqueness and regularity of solutions to problems belonging to a class of second order in time semilinear partial differential equations
 in Banach spaces can be found in \cite{Carvalho}. In \cite{Carvalho}, the authors study the analyticity of semigroups generated by a class of operators in the $L^{p}$ framework and obtained local existence and regularity results for some second order (wave like) semilinear problems of parabolic nature. We also refer the readers to \cite{fanli} for the existence and exponential stability issues for elastic systems with structural damping in Banach spaces. For further references regarding the well posedness issues of damped plate equation we refer to \cite{ebert} and \cite{veraar}. The readers can also consult \cite{Mitraexistence}, \cite{raymondbeam} and \cite{veiga} for the application of the regularity results of the damped Euler-Bernoulli beam equation in studying the well posedness of coupled dynamical systems and more particularly fluid structure interaction problems.\\
  
  To the best of our knowledge the present article is the first one in the literature obtaining a Carleman estimate for the adjoint of the operator \eqref{formaloperator}. Using spectral methods the null controllability of the system \eqref{dampedbeam} with $a=0$ is studied in \cite{lasieckadampedbeam}, \cite{julienedward} and \cite{millerbeam}. There exist several articles dealing with the controllability of undamped plate equation. The exact controllability problem using boundary controls of an undamped Euler-Bernoulli beam equation is considered in \cite{lasieckabeam}. In \cite{lasieckabeam} the authors prove the exact controllability result by proving an observability inequality for the homogeneous boundary value adjoint system using multiplier method. Exact controllability of a Euler-Bernoulli beam with variable coefficients with semi internal control is studied in \cite{JUKIM}. For controllability results of thin plate and beam equations one can also consult \cite{lagnese}. For the controllability issues of a coupled parabolic-hyperbolic dynamics involving an elastic structure, for instance thermoelastic systems, one can look into the articles \cite{lebeauzuazua} and \cite{avalos2}.\\  
  
  We would also like to quote the articles \cite{zhangxu} and \cite{fuxiaoyu} for the use of Carleman estimates in order to prove controllability results for plate equations. In \cite{zhangxu} the author considers the exact controllability problem of a semilinear plate equation with superlinear nonlinearity while in \cite{fuxiaoyu} the author deals with a linear plate equation with potential. In \cite{zhangxu} the author obtains a Carleman estimate by decomposing the plate operator into two Schrödinger operators while in \cite{fuxiaoyu} the author derives a Carleman estimate directly without using Schrödinger operators. We would like to point out that the Carleman weights used in \cite{zhangxu} and \cite{fuxiaoyu} completely differ from that of ours, introduced in Section \ref{w8fn}. This is because the linearized operators in \cite{zhangxu} and \cite{fuxiaoyu} are of hyperbolic nature whereas due to the structural damping the system \eqref{dampedbeam} is parabolic.\\
  
  The study of Carleman estimate for a parabolic equation involving fourth order space derivative is quite recent in the literature. The article \cite{cerpamercado} establishes the first Carleman estimate for a parabolic equation in dimension one involving fourth order derivative in space. In \cite{cerpamercado} the authors study the local exact controllability to the trajectories of the Kuramoto-Sivashinsky equation with boundary controls using Carleman estimate. For Carleman estimate and its application to the controllability of similar fourth order parabolic equations in dimension one we also refer the readers to \cite{zhouobsinq4thorder} and \cite{Carrenocerpa}. In dimension $N\geqslant 2,$ Carleman estimate for a fourth order parabolic equation is established in a very recent article \cite{GuerrKb1}. Our system \eqref{dampedbeam} is fourth order in space, second order in time and further involves a damping term $\partial_{txx}\beta$ and hence it is quite different from the models considered in \cite{GuerrKb1}, \cite{Carrenocerpa}, \cite{zhouobsinq4thorder} and \cite{cerpamercado} which are first order in time. To the best of our knowledge the present article is the first one proving a Carleman estimate for a parabolic equation which is fourth order in space and second order in time.
  \subsection{Outline} In Section \ref{Carlemanestimate} we prove Theorem \ref{Carlbeamthm}, the central result of this article and further state a Corollary \ref{corpotential} which can be readily obtained as a consequence of Theorem \ref{Carlbeamthm}. Next in Section \ref{Nullcontrollability} we prove Theorem \ref{Centraltheoremnullcont} as an application of the Carleman estimate  proved in Corollary \ref{corpotential}. In Section \ref{appendix} we include the proof of Lemma \ref{existencepotential}.   
 \section{Proof of Theorem \ref{Carlbeamthm} and a corollary}\label{Carlemanestimate}
 From now on until the end of this article we fix the controllability time $T.$\\
At this moment we can recall the definition of the weight functions $\phi$ and $\xi$ which were introduced in \ref{w8fn}. In our computations afterwards we will frequently use the following estimates, valid on ${\mathbb{T}^{d}_{L}} \times (0,T)$:
\begin{equation}\label{prilies}
%\left\{
\begin{array}{l}
\ds |\partial^{(i)}_{x}\phi|\leqslant C\lambda^{i}{\xi} \quad \mbox{for all}\,i\in\{1,2,3,4\},
\\
\ds |\partial_{t}\phi|\leqslant C{\xi}^{3/2},\quad 
|\partial_{tt}\phi|\leqslant C{\xi}^{2},\quad
|\partial_{tx}\phi|\leqslant C\lambda{\xi}^{3/2},\quad 
|\partial_{txx}\phi|\leqslant C\lambda^{2}{\xi}^{3/2},
\\
\ds	|\partial_{txxx}\phi|\leqslant C\lambda^{3}{\xi}^{3/2},\quad
|\partial_{ttx}\phi|\leqslant C\lambda{\xi}^{2}\quad \mbox{and}\quad|\partial_{ttxx}\phi|\leqslant C\lambda^{2}{\xi}^{2},
\end{array}
%\right.
\end{equation}
and 
\begin{equation}\label{prilies*}
%\left\{ 
\begin{array}{l}
\ds |\partial^{(i)}_{x}{\xi}|\leqslant C\lambda^{i}{\xi}\quad\mbox{for all}\,\, i\in\{1,2,3,4\},\\
\ds 	|\partial_{t}{\xi}|\leqslant C{\xi}^{3/2},\quad
|\partial_{tt}{\xi}|\leqslant C{\xi}^{2},\quad
|\partial_{tx}{\xi}|\leqslant C\lambda{\xi}^{3/2},\quad
|\partial_{txx}{\xi}|\leqslant C\lambda^{2}{\xi}^{3/2}
\\
\ds 	|\partial_{txxx}{\xi}|\leqslant C\lambda^{3}{\xi}^{3/2},\quad
|\partial_{ttx}\xi|\leqslant C\lambda{\xi}^{2}\quad
\mbox{and}\quad |\partial_{ttxx}\xi|\leqslant C\lambda^{2}{\xi}^{2},
\end{array}
%\right.
\end{equation}
and, for $\lambda$ large enough, for all $(x,t) \in [0,d] \times (0,T)$ and $i \in \{1, 2, 3, 4\}$, 
\begin{equation}
\label{Positivity-Weights}
-\partial^{(i)}_{x}\phi  = \partial^{(i)}_{x}\xi \geqslant c\lambda^{i}{\xi}.
\end{equation}
\subsection{Carleman estimate for an adjoint damped beam equation}\label{Carldbeam}
In the following we prove Theorem \ref{Carlbeamthm} which corresponds to the Carleman estimate for the adjoint of the damped beam equation. 
%The formal adjoint of the primal problem \eqref{dampedbeam} is given as follows:
% \begin{equation}\label{adjbeam1}
%\left\{ \begin{array}{ll}
%{\partial_{tt}{\psi}}+{\partial_{txx}{\psi}}+{\partial_{xxxx}{\psi}}=f_\psi \,\,& \mbox{in}\,\, {\mathbb{T}^{d}_{L}}\times(0,T),
%\\
%{\psi}(.,T)={\psi}_{T},\quad{\partial_{t}{\psi}}(.,T)={\psi}_{T}^{1}\,\,&\mbox{in}\,\, {\mathbb{T}^{d}_{L}}.
%\end{array}\right.
%\end{equation}
\begin{proof}[Proof of Theorem \ref{Carlbeamthm}]
	In the proof for simplicity of notations we will write:
	\begin{equation}\label{adjbeam1}
	\begin{array}{l}
	\displaystyle f_{\psi}=\partial_{tt}\psi+\partial_{txx}\psi+\partial_{xxxx}\psi.
	\end{array}
	\end{equation}
	We introduce the change of unknown
	$$w=e^{-s\phi}{\psi}.$$
	In view of \eqref{adjbeam1}, $w$ satisfies:
	\begin{equation}\label{eqw}
	\begin{split}
	e^{-s\phi}{f_{\psi}}&=e^{-s\phi}({\partial_{tt}{\psi}}+{\partial_{txx}{\psi}}+{\partial_{xxxx}{\psi}}+a\psi)\\
	&=e^{-s\phi}(\partial_{tt}(e^{s\phi}w)+\partial_{txx}(e^{s\phi}w)+\partial_{xxxx}(e^{s\phi}w))=P_{\phi}w.
	\end{split}
	\end{equation}
	We write $P_{\phi}w$ in the form:
	\begin{equation}\label{}
	\begin{array}{l}
	P_{\phi}w=P_{1}w+P_{2}w+\mathit{R}w,
	\end{array}
	\end{equation}
	where 
	\begin{equation}\label{p12r}
	\left\{ \begin{array}{ll}
	P_{1}w=&s^{4}(\partial_{x}\phi)^{4}w
	+6 s^{2}(\partial_{x}\phi)^{2}\partial_{xx}w+\partial_{xxxx}w
	+ 2 s\partial_{x}\phi\partial_{xt}w
	+\partial_{tt}w,\\
	P_{2}w=&4s^{3}(\partial_{x}\phi)^{3}\partial_{x}w+4 s\partial_{x}\phi\partial_{xxx}w+\partial_{xxt}w+ s^{2}(\partial_{x}\phi)^{2}\partial_{t}w\\
	&+6(1+\zeta)s^{3}(\partial_{x}\phi)^{2}\partial_{xx}\phi w,\\
	\mathit{R}w=& s^{2}(\partial_{t}\phi)^{2}w+s\partial_{t}\phi\partial_{xx}w
	+ s^{3}\partial_{t}\phi(\partial_{x}\phi)^{2}w+s\partial_{t}\phi\partial_{t}w+2s^{2}\partial_{t}\phi\partial_{x}\phi\partial_{x}w\\
	&+\frac{s}{2}\partial_{tt}\phi w-s\partial_{xxt}\phi w+2 s\partial_{xt}\phi\partial_{x}w+4s^{2}\partial_{x}\phi\partial_{xxx}\phi w + s\partial_{xx}\phi\partial_{t}w\\
	&+12s^{2}\partial_{x}\phi\partial_{xx}\phi\partial_{x}w+3 s^{2}(\partial_{xx}\phi)^{2}w+\frac{s}{2}\partial_{tt}\phi w+2 s^{2}\partial_{xt}\phi\partial_{x}\phi w\\
	&+s^{2}\partial_{t}\phi\partial_{xx}\phi w+6 s\partial_{xx}\phi\partial_{xx}w+ s\partial_{xxxx}\phi w + 4s\partial_{xxx}\phi\partial_{x}w\\
	&-6\zeta s^{3}(\partial_{x}\phi)^{2}\partial_{xx}\phi w,
	\end{array}\right.
	\end{equation}
	where $\zeta$ is a free parameter which will be fixed later.\\	
	Based on the identity 
	$$P_{1}w+P_{2}w={f_{\psi}}e^{-s\phi}-\mathit{R}w,$$
	we obtain
	\begin{equation}\label{trieq}
	\begin{split}
	\displaystyle
	\int_0^T \int_{\mathbb{T}_L}|P_{1}w|^{2}+\int_0^T \int_{\mathbb{T}_L}|P_{2}w|^{2}&+2\int_0^T \int_{\mathbb{T}_L} P_{1}wP_{2}w=\int_0^T \int_{\mathbb{T}_L}|{f_{\psi}}e^{-s\phi}-\mathit{R}w|^{2}\\
	&\leqslant 2\int_0^T \int_{\mathbb{T}_L} |{f_{\psi}}|^{2}e^{-2s\phi}+2\int_0^T \int_{\mathbb{T}_L}|\mathit{R}w|^{2}.
	\end{split}
	\end{equation} 
	The crucial point is to obtain suitable estimates for the product term $\displaystyle \int_0^T \int_{\mathbb{T}_L} P_{1}wP_{2}w.$ 
	We will denote by $I_{i,j}$ the cross product of the $i$-th term of $P_1 w$ and of the $j$-th term of $P_2 w$, so that 
	$$\int_0^T \int_{\mathbb{T}_L}P_{1}wP_{2}w=\sum\limits_{i,j=1}^{i=5,j=5}I_{ij}.$$
	In the following estimates to make the presentation simpler we will write L.O.T (lower order terms) for the terms which are small (for large values of the parameters $s$ and $\lambda$) with respect to the left hand side of \eqref{crlestbmthm}, i.e. for which there exists a constant $C$ independent of $s$ and $\lambda$ such that
	\begin{multline*}
	| L.O.T | 
	\leq 
	C \left(\frac{1}{s} + \frac{1}{\lambda}\right)
	\left(
	s^{7}\lambda^{8}\int_0^T \int_{\mathbb{T}_L} {\xi}^{7}|{\psi}|^{2}e^{-2s\phi}+s^{5}\lambda^{6}\int_0^T \int_{\mathbb{T}_L} {\xi}^{5}|{\partial_{x}{\psi}}|^{2}e^{-2s\phi}
	\right.\\ 
	\left.
	+s^{3}\lambda^{4}\int_0^T \int_{\mathbb{T}_L} {\xi}^{3} ( |{\partial_{xx}{\psi}}|^{2}+ |{\partial_{t}{\psi}}|^{2}) e^{-2s\phi}
	\displaystyle+s\lambda^{2}\int_0^T \int_{\mathbb{T}_L} {\xi}( |{\partial_{tx}{\psi}}|^{2}+|{\partial_{xxx}{\psi}}|^{2})e^{-2s\phi} 
	\right).
	\end{multline*} 
	In particular, note that we immediately get that 
	\begin{equation}
	\label{R-is-LOT}
	\int_0^T \int_{\mathbb{T}_L} | R w|^2 \leqslant  L.O.T.
	\end{equation}

	We list below the computations of each $I_{ij}$.
	\begin{equation}\label{I11}
	\begin{split}
	I_{11}=4s^{7}\int_0^T \int_{\mathbb{T}_L}(\partial_{x}\phi)^{7}w\partial_{x}w
	%=2s^{7}\int_0^T \int_{\mathbb{T}_L}(\partial_{x}\phi)^{7}\partial_{x}|w|^{2}
	=-14s^{7}\int_0^T \int_{\mathbb{T}_L}(\partial_{x}\phi)^{6}\partial_{xx}\phi w^{2}.
	\end{split}
	\end{equation}
	\begin{equation}\label{I12}
	\begin{split}
	I_{12}&=4s^{5}\int_0^T \int_{\mathbb{T}_L}(\partial_{x}\phi)^{5}w\partial_{xxx}w
	%\\
	%&=-4s^{5}\int_0^T \int_{\mathbb{T}_L}\partial_{x}((\partial_{x}\phi)^{5}w)\partial_{xx}w\\
	% &=-20s^{5}\int_0^T \int_{\mathbb{T}_L}(\partial_{x}\phi)^{4}\partial_{xx}\phi w\partial_{xx}w
	% 		-2s^{5}\int_0^T \int_{\mathbb{T}_L}(\partial_{x}\phi)^{5}\partial_{x}(\partial_{x}w)^{2}\\
	=-120s^{5}\int_0^T \int_{\mathbb{T}_L}(\partial_{x}\phi)^{2}(\partial_{xx}\phi)^{3}w^{2}
	\\
	&\quad -80s^{5}\int_0^T \int_{\mathbb{T}_L}(\partial_{x}\phi)^{3}\partial_{xx}\phi\partial_{xxx}\phi w^{2}-40s^{5}\int_0^T \int_{\mathbb{T}_L} (\partial_{x}\phi)^{3}\partial_{xx}\phi\partial_{xxx}\phi w^{2}
	\\
	&
	\quad-10s^{5}\int_0^T \int_{\mathbb{T}_L}(\partial_{x}\phi)^{4}\partial_{xxxx}\phi w^{2}+30s^{5}\int_0^T \int_{\mathbb{T}_L}(\partial_{x}\phi)^{4}\partial_{xx}\phi(\partial_{x}w)^{2}
	\\
	&
	=L.O.T+30s^{5}\int_0^T \int_{\mathbb{T}_L}(\partial_{x}\phi)^{4}\partial_{xx}\phi(\partial_{x}w)^{2}.
	\end{split}
	\end{equation}
	\begin{equation}\label{I13}
	\begin{split}
	I_{13}&= s^{4}\int_0^T \int_{\mathbb{T}_L}(\partial_{x}\phi)^{4}w\partial_{xxt}w
	\\
	% 	&=(-4s^{4}\int_0^T \int_{\mathbb{T}_L}(\partial_{x}\phi)^{3}\partial_{xx}\phi w\partial_{xt}w-s^{4}\int_0^T \int_{\mathbb{T}_L}(\partial_{x}\phi)^{4}\partial_{x}w\partial_{xt}w)\\
	&=-12s^{4}\int_0^T \int_{\mathbb{T}_L}(\partial_{x}\phi)\partial_{xt}\phi(\partial_{xx}\phi)^{2}w^{2}-12s^{4}\int_0^T \int_{\mathbb{T}_L}(\partial_{x}\phi)^{2}\partial_{xx}\phi\partial_{xxt}\phi w^{2}\\
	&\quad-6s^{4}\int_0^T \int_{\mathbb{T}_L}(\partial_{x}\phi)^{2}\partial_{xt}\phi\partial_{xxx}\phi w^{2}
	\quad -2s^{4}\int_0^T \int_{\mathbb{T}_L} (\partial_{x}\phi)^{3}\partial_{txxx}\phi w^{2}\\
	&\quad+2s^{4}\int_0^T \int_{\mathbb{T}_L} (\partial_{x}\phi)^{3}\partial_{tx}\phi(\partial_{x}w)^{2}
	+4s^{4}\int_0^T \int_{\mathbb{T}_L}(\partial_{x}\phi)^{3}\partial_{xx}\phi\partial_{x}w\partial_{t}w\\
	&=L.O.T+4 s^{4}\int_0^T \int_{\mathbb{T}_L}(\partial_{x}\phi)^{3}\partial_{xx}\phi\partial_{x}w\partial_{t}w.
	\end{split}
	\end{equation}
	
	\begin{equation}
	\begin{array}{l}
	\displaystyle
	I_{14}=3 s^{6}\int_0^T \int_{\mathbb{T}_L}(\partial_{x}\phi)^{5}\partial_{tx}\phi w^{2}=L.O.T.
	\end{array}
	\end{equation}
	\begin{equation}
	\begin{array}{l}
	\displaystyle
	I_{15}=6(1+\zeta)s^{7}\int_0^T \int_{\mathbb{T}_L} (\partial_{x}\phi)^{6}\partial_{xx}\phi w^{2}.
	\end{array}
	\end{equation}
	\begin{equation}
	\begin{array}{l}
	\displaystyle
	I_{21}=-60s^{5}\int_0^T \int_{\mathbb{T}_L}(\partial_{x}\phi)^{4}\partial_{xx}\phi(\partial_{x}w)^{2}.
	\end{array}
	\end{equation}
	\begin{equation}
	\begin{array}{l}
	\displaystyle
	I_{22}=-36s^{3}\int_0^T \int_{\mathbb{T}_L}(\partial_{x}\phi)^{2}\partial_{xx}\phi(\partial_{xx}w)^{2}.
	\end{array}
	\end{equation}
	\begin{equation}
	\begin{array}{l}
	\displaystyle
	I_{23}=-12 s^{2}\int_0^T \int_{\mathbb{T}_L}(\partial_{x}\phi)\partial_{xt}\phi(\partial_{xx}w)^{2}=L.O.T.
	\end{array}
	\end{equation}
	
	\begin{equation}
	\begin{split}
	I_{24}&=6 s^{4}\int_0^T \int_{\mathbb{T}_L}(\partial_{x}\phi)^{4}\partial_{xx}w\partial_{t}w\\
	% 	&=(-24s^{4}\int_0^T \int_{\mathbb{T}_L}(\partial_{x}\phi)^{3}\partial_{xx}\phi\partial_{x}w\partial_{t}w-6s^{4}\int_0^T \int_{\mathbb{T}_L} (\partial_{x}\phi)^{4}\partial_{x}w\partial_{tx}w)\\
	&=-24s^{4}\int_0^T \int_{\mathbb{T}_L}(\partial_{x}\phi)^{3}\partial_{xx}\phi\partial_{x}w\partial_{t}w+12s^{4}\int_0^T \int_{\mathbb{T}_L} (\partial_{x}\phi)^{3}\partial_{xt}\phi(\partial_{x}w)^{2}\\
	&=L.O.T-24s^{4}\int_0^T \int_{\mathbb{T}_L} (\partial_{x}\phi)^{3}\partial_{xx}\phi\partial_{x}w\partial_{t}w.
	\end{split}
	\end{equation}
	
	\begin{align}
	I_{25}&=36(1+\zeta)s^{5}\int_0^T \int_{\mathbb{T}_L}(\partial_{x}\phi)^{4}\partial_{xx}\phi w\partial_{xx}w
	= (1+\zeta)\left(216s^{5}\int_0^T \int_{\mathbb{T}_L}(\partial_{x}\phi)^{2}(\partial_{xx}\phi)^{3}w^{2}
	\right.\notag
	\\
	& \left. +144s^{5}\int_0^T \int_{\mathbb{T}_L}(\partial_{x}\phi)^{3}\partial_{xx}\phi\partial_{xxx}\phi w^{2}
	\right.
	\left. +72s^{5}\int_0^T \int_{\mathbb{T}_L} (\partial_{x}\phi)^{3}\partial_{xx}\phi\partial_{xxx}\phi w^{2}
	\right.
	\\
	&
	\left.
	+18s^{5}\int_0^T \int_{\mathbb{T}_L}(\partial_{x}\phi)^{4}\partial_{xxxx}\phi w^{2} -36s^{5}\int_0^T \int_{\mathbb{T}_L} (\partial_{x}\phi)^4\partial_{xx}\phi(\partial_{x}w)^{2}\right)\notag\\
	&=L.O.T-36(1+\zeta)s^{5}\int_0^T \int_{\mathbb{T}_L} (\partial_{x}\phi)^{4}\partial_{xx}\phi(\partial_{x}w)^{2}.\notag
	\end{align}
	\begin{equation}
	\begin{split}
	I_{31}&=4s^{3}\int_0^T \int_{\mathbb{T}_L}(\partial_{x}\phi)^{3}\partial_{x}w\partial_{xxxx}w\\
	&=-12s^{3}\int_0^T \int_{\mathbb{T}_L}(\partial_{x}\phi)^{2}\partial_{xx}\phi\partial_{x}w\partial_{xxx}w-4s^{3}\int_0^T \int_{\mathbb{T}_L}(\partial_{x}\phi)^{3}\partial_{xx}w\partial_{xxx}w\\
	%&=24s^{3}\int_0^T \int_{\mathbb{T}_L} (\partial_{x}\phi)(\partial_{xx}\phi)^{2}\partial_{x}w\partial_{xx}w+12s^{3}\int_0^T \int_{\mathbb{T}_L}(\partial_{x}\phi)^{2}\partial_{xxx}\phi\partial_{x}w\partial_{xx}w\\
	%&\quad+12s^{3}\int_0^T \int_{\mathbb{T}_L}(\partial_{x}\phi)^{2}\partial_{xx}\phi(\partial_{xx}w)^{2}
	%+6s^{3}\int_0^T \int_{\mathbb{T}_L}(\partial_{x}\phi)^{2}\partial_{xx}\phi(\partial_{xx}w)^{2}\\
	&=L.O.T+18s^{3}\int_0^T \int_{\mathbb{T}_L}(\partial_{x}\phi)^{2}\partial_{xx}\phi(\partial_{xx}w)^{2}.
	\end{split}
	\end{equation}
	\begin{equation}
	\begin{array}{l}
	\displaystyle
	I_{32}=-2s\int_0^T \int_{\mathbb{T}_L}\partial_{xx}\phi(\partial_{xxx}w)^{2}.
	\end{array}
	\end{equation}
	\begin{equation}
	\begin{array}{l}
	\displaystyle
	I_{33}=\int_0^T \int_{\mathbb{T}_L}\partial_{xxxx}w\partial_{xxt}w=0.
	\end{array}
	\end{equation}
	\begin{equation}
	\begin{split}
	I_{34}&=s^{2}\int_0^T \int_{\mathbb{T}_L}(\partial_{x}\phi)^{2}\partial_{xxxx}w\partial_{t}w\\
	&=-2s^{2}\int_0^T \int_{\mathbb{T}_L}(\partial_{x}\phi)\partial_{xx}\phi\partial_{xxx}w\partial_{t}w-s^{2}\int_0^T \int_{\mathbb{T}_L} (\partial_{x}\phi)^{2}\partial_{xxx}w\partial_{tx}w\\
	%&=2s^{2}\int_0^T \int_{\mathbb{T}_L}(\partial_{xx}\phi)^{2}\partial_{xx}w\partial_{t}w+2s^{2}\int_0^T \int_{\mathbb{T}_L}\partial_{x}\phi\partial_{xxx}\phi\partial_{xx}w\partial_{t}w\\
	%&\quad+2s^{2}\int_0^T \int_{\mathbb{T}_L}\partial_{x}\phi\partial_{xx}\phi\partial_{xx}w\partial_{tx}w
	%+2s^{2}\int_0^T \int_{\mathbb{T}_L} \partial_{x}\phi\partial_{xx}\phi\partial_{xx}w\partial_{tx}w\\
	%&\quad-s^{2}\int_0^T \int_{\mathbb{T}_L}\partial_{x}\phi\partial_{tx}\phi(\partial_{xx}w)^{2}\\
	&=L.O.T+4s^{2}\int_0^T \int_{\mathbb{T}_L}\partial_{x}\phi\partial_{xx}\phi\partial_{xx}w\partial_{tx}w.
	\end{split}
	\end{equation}
	
	\begin{align}
	I_{35}&=6(1+\zeta)s^{3}\int_0^T \int_{\mathbb{T}_L}(\partial_{x}\phi)^{2}\partial_{xx}\phi\partial_{xxxx}ww\notag\\
	&=-12(1+\zeta)s^{3}\int_0^T \int_{\mathbb{T}_L} \partial_{x}\phi(\partial_{xx}\phi)^{2}\partial_{xxx}ww
	-6(1+\zeta)s^{3}\int_0^T \int_{\mathbb{T}_L}(\partial_{x}\phi)^{2}\partial_{xxx}\phi\partial_{xxx}ww\notag\\
	&-6(1+\zeta)s^{3}\int_0^T \int_{\mathbb{T}_L}(\partial_{x}\phi)^{2}\partial_{xx}\phi\partial_{xxx}w\partial_{x}w\\
	&=L.O.T+6(1+\zeta)s^{3}\int_0^T \int_{\mathbb{T}_L}(\partial_{x}\phi)^{2}\partial_{xx}\phi(\partial_{xx}w)^{2}.\notag
	\end{align}
	\begin{equation}
	I_{41}=-16s^{4}\int_0^T \int_{\mathbb{T}_L}(\partial_{x}\phi)^{3}\partial_{tx}\phi(\partial_{x}w)^{2}=L.O.T.
	\end{equation}
	\begin{equation}
	\begin{split}
	I_{42}&=8s^{2}\int_0^T \int_{\mathbb{T}_L}(\partial_{x}\phi)^{2}\partial_{xt}w\partial_{xxx}w\\
	% 		&=-16s^{2}\int_0^T \int_{\mathbb{T}_L}\partial_{x}\phi\partial_{xx}\phi\partial_{xt}w\partial_{xx}w-8s^{2}\int_0^T \int_{\mathbb{T}_L}(\partial_{x}\phi)^{2}\partial_{xxt}w\partial_{xx}w\\
	&=-16s^{2}\int_0^T \int_{\mathbb{T}_L}\partial_{x}\phi\partial_{xx}\phi\partial_{xt}w\partial_{xx}w+8s^{2}\int_0^T \int_{\mathbb{T}_L}\partial_{x}\phi\partial_{xt}\phi(\partial_{xx}w)^{2}\\
	&=L.O.T-16s^{2}\int_0^T \int_{\mathbb{T}_L}\partial_{x}\phi\partial_{xx}\phi\partial_{xt}w\partial_{xx}w.
	\end{split}
	\end{equation}
	\begin{equation}
	\begin{array}{l}
	\displaystyle
	I_{43}=-s\int_0^T \int_{\mathbb{T}_L}\partial_{xx}\phi(\partial_{xt}w)^{2}.
	\end{array}
	\end{equation}
	\begin{equation}
	\begin{array}{l}
	\displaystyle
	I_{44}=-3s^{3}\int_0^T \int_{\mathbb{T}_L}(\partial_{x}\phi)^{2}\partial_{xx}\phi(\partial_{t}w)^{2}.
	\end{array}
	\end{equation}
	
	\begin{align}
	I_{45}&=12(1+\zeta)s^{4}\int_0^T \int_{\mathbb{T}_L}(\partial_{x}\phi)^{3}\partial_{xx}\phi\partial_{xt}ww\notag\\
	&=(1+\zeta)\left(-36s^{4}\int_0^T \int_{\mathbb{T}_L}(\partial_{x}\phi)^{2}(\partial_{xx}\phi)^{2}\partial_{t}ww
	-12s^{4}\int_0^T \int_{\mathbb{T}_L}(\partial_{x}\phi)^{3}\partial_{xxx}\phi\partial_{t}ww
	\right.\notag
	\\
	&\quad
	\left.
	-12s^{4}\int_0^T \int_{\mathbb{T}_L}(\partial_{x}\phi)^{3}\partial_{xx}\phi\partial_{t}w\partial_{x}w\right)\\
	&=L.O.T-12(1+\zeta)s^{4}\int_0^T \int_{\mathbb{T}_L}(\partial_{x}\phi)^{3}\partial_{xx}\phi\partial_{t}w\partial_{x}w.\notag
	\end{align}
	\begin{equation}
	\begin{split}
	I_{51}&=4s^{3}\int_0^T \int_{\mathbb{T}_L}(\partial_{x}\phi)^{3}\partial_{x}w\partial_{tt}w\\
	&=-12s^{3}\int_0^T \int_{\mathbb{T}_L}(\partial_{x}\phi)^{2}\partial_{xt}\phi\partial_{x}w\partial_{t}w-4s^{3}\int_0^T \int_{\mathbb{T}_L}(\partial_{x}\phi)^{3}\partial_{xt}w\partial_{t}w\\
	%&=-12s^{3}\int_0^T \int_{\mathbb{T}_L}(\partial_{x}\phi)^{2}\partial_{xt}\phi\partial_{x}w\partial_{t}w+6s^{3}\int_0^T \int_{\mathbb{T}_L}(\partial_{x}\phi)^{2}\partial_{xx}\phi(\partial_{t}w)^{2}\\
	&=L.O.T+6s^{3}\int_0^T \int_{\mathbb{T}_L}(\partial_{x}\phi)^{2}\partial_{xx}\phi(\partial_{t}w)^{2}.
	\end{split}
	\end{equation}
	
	\begin{equation}
	\begin{split}
	I_{52}&=4s\int_0^T \int_{\mathbb{T}_L}\partial_{x}\phi\partial_{xxx}w\partial_{tt}w\\
	&=-4s\int_0^T \int_{\mathbb{T}_L}\partial_{xt}\phi\partial_{xxx}w\partial_{t}w-4s\int_0^T \int_{\mathbb{T}_L}\partial_{x}\phi\partial_{xxxt}w\partial_{t}w\\
	%&=-4s\int_0^T \int_{\mathbb{T}_L}\partial_{xt}\phi\partial_{xxx}w\partial_{t}w+4s\int_0^T \int_{\mathbb{T}_L}\partial_{xx}\phi\partial_{xxt}w\partial_{t}w+4s\int_0^T \int_{\mathbb{T}_L}\partial_{x}\phi\partial_{xxt}w\partial_{xt}w\\
	%&=L.O.T-4s\int_0^T \int_{\mathbb{T}_L}\partial_{xx}\phi(\partial_{xt}w)^{2}-2s\int_0^T \int_{\mathbb{T}_L}\partial_{xx}\phi(\partial_{xt}w)^{2}-4s\int_0^T \int_{\mathbb{T}_L}\partial_{xxx}\phi\partial_{xt}w\partial_{t}w\\
	&=L.O.T-6s\int_0^T \int_{\mathbb{T}_L}\partial_{xx}\phi(\partial_{xt}w)^{2}.
	\end{split}
	\end{equation}
	
	\begin{equation}
	\begin{array}{l}
	\displaystyle
	I_{53}=\int_0^T \int_{\mathbb{T}_L}\partial_{tt}w\partial_{xxt}w=\int_0^T \int_{\mathbb{T}_L}\partial_{t}(\partial_{tx}w)^{2}=0.
	\end{array}
	\end{equation}
	\begin{equation}
	\begin{split}
	\displaystyle
	I_{54}&=s^{2}\int_0^T \int_{\mathbb{T}_L}(\partial_{x}\phi)^{2}\partial_{t}w\partial_{tt}w\\
	%&=\frac{1}{2}\int_0^T \int_{\mathbb{T}_L}(\partial_{x}\phi)^{2}\partial_{t}(\partial_{t}w)^{2}\\
	&=-s^{2}\int_0^T \int_{\mathbb{T}_L}\partial_{x}\phi\partial_{tx}\phi(\partial_{t}w)^{2}=L.O.T.
	\end{split}
	\end{equation}
	
	\begin{align}
	I_{55}&=6(1+\zeta)s^{3}\int_0^T \int_{\mathbb{T}_L}(\partial_{x}\phi)^{2}\partial_{xx}\phi\partial_{tt}ww =(1+\zeta)\left(-6s^{3}\int_0^T \int_{\mathbb{T}_L}\partial_{x}\phi\partial_{tx}\phi\partial_{xx}\phi\partial_{t}(w^{2})
	\right.\nonumber
	\\
	&\left. -3s^{3}\int_0^T \int_{\mathbb{T}_L}(\partial_{x}\phi)^{2}\partial_{txx}\phi\partial_{t}(w^{2})
	-6s^{3}\int_0^T \int_{\mathbb{T}_L}(\partial_{x}\phi)^{2}\partial_{xx}\phi(\partial_{t}w)^{2}\right)\\
	%&=(1+\zeta)(6s^{3}\int_0^T \int_{\mathbb{T}_L}(\partial_{tx}\phi)^{2}\partial_{xx}\phi w^{2}+6s^{3}\int_0^T \int_{\mathbb{T}_L}\partial_{x}\phi\partial_{ttx}\phi\partial_{xx}\phi w^{2}\\
	%&\quad +6s^{3}\int_0^T \int_{\mathbb{T}_L}\partial_{x}\phi\partial_{tx}\phi\partial_{txx}\phi w^{2}+6s^{3}\int_0^T \int_{\mathbb{T}_L}\partial_{x}\phi\partial_{tx}\phi\partial_{txx}\phi w^{2}\\
	%& \quad+3s^{3}\int_0^T \int_{\mathbb{T}_L}(\partial_{x}\phi)^{2}\partial_{ttxx}\phi w^{2}-6s^{3}\int_0^T \int_{\mathbb{T}_L} (\partial_{x}\phi)^{2}\partial_{xx}\phi(\partial_{t}w)^{2})\\
	&=L.O.T-6(1+\zeta)s^{3}\int_0^T \int_{\mathbb{T}_L}(\partial_{x}\phi)^{2}\partial_{xx}\phi(\partial_{t}w)^{2}.\notag
	\end{align}
	Hence we find that
	\begin{align}\label{estsum}
	&\int_0^T \int_{\mathbb{T}_L}P_{1}wP_{2}w
	=\sum\limits_{i,j=1}^{i=5,j=5}I_{ij}\notag
	\\ 
	&=(-8+6\zeta)s^{7}\int_0^T \int_{\mathbb{T}_L}(\partial_{x}\phi)^{6}\partial_{xx}\phi w^{2}
	+(-66-36\zeta)s^{5}\int_0^T \int_{\mathbb{T}_L}(\partial_{x}\phi)^{4}\partial_{xx}\phi(\partial_{x}w)^{2}\notag
	\\
	&
	+(-12+6\zeta)s^{3}\int_0^T \int_{\mathbb{T}_L}(\partial_{x}\phi)^{2}\partial_{xx}\phi(\partial_{xx}w)^{2}
	+(-3-6\zeta)s^{3}\int_0^T \int_{\mathbb{T}_L}(\partial_{x}\phi)^{2}\partial_{xx}\phi(\partial_{t}w)^{2}\notag
	\\
	&
	-2s\int_0^T \int_{\mathbb{T}_L}\partial_{xx}\phi(\partial_{xxx}w)^{2}
	-7s\int_0^T \int_{\mathbb{T}_L}\partial_{xx}\phi(\partial_{xt}w)^{2}
	\\
	&
	+ (-32-12\zeta)s^{4}\int_0^T \int_{\mathbb{T}_L}(\partial_{x}\phi)^{3}\partial_{xx}\phi\partial_{x}w\partial_{t}w
	-12s^{2}\int_0^T \int_{\mathbb{T}_L}\partial_{x}\phi\partial_{xx}\phi\partial_{xt}w\partial_{xx}w+L.O.T\notag\\
	&=\sum\limits_{n=1}^{8}E_{n}+L.O.T.\notag
	\end{align}
	Now, we adjust the parameter $\zeta$ such that all the coefficients of $E_n$ for $n \in \{1, \cdots, 6\}$ are negative and the terms $E_{7}$ and $E_{8}$ can be absorbed by using $E_{n}$ for $n\in\{1,...,6\}.$\\
	%
	% 	Now we will choose the range of the values of $\zeta$ in such a way that the coefficients of $E_{n},$ $\forall 1\leqslant n\leqslant 6,$ have negative signs (in that case \eqref{prilies}$_{3}$ gives the positivity of $E_{n},$ $\forall 1\leqslant n\leqslant 6$) and the terms $E_{7}$ and $E_{8}$ can be estimated by using $E_{n},$ $\forall 1\leqslant n\leqslant 6.$ 
	In that direction we observe that, according to Young's inequality, for $\alpha_1$ and $\alpha_2$ positive, 
	\begin{align}\label{E7}
	|E_{7}|&=(32+12\zeta)s^{4}|\int_0^T \int_{\mathbb{T}_L}(\partial_{x}\phi)^{3}\partial_{xx}\phi\partial_{x}w\partial_{t}w|\notag\\
	&\leqslant\frac{(32+12\zeta)}{2\alpha_{1}}s^{5}\int_0^T \int_{\mathbb{T}_L}|(\partial_{x}\phi)^{4}\partial_{xx}\phi(\partial_{x}w)^{2}|
	+\frac{(32+12\zeta)\alpha_{1}}{2}s^{3}\int_0^T \int_{\mathbb{T}_L}|(\partial_{x}\phi)^{2}\partial_{xx}\phi(\partial_{t}w)^{2}|\notag
	\\
	& \leq \frac{(32+12\zeta)}{2\alpha_{1}|66 + 36 \zeta|} |E_2|
	+ 
	\frac{(32+12\zeta)\alpha_{1}}{2 |3 + 6 \zeta|} |E_4|
	,
	\end{align}
	and 
	\begin{equation}\label{E8}
	\begin{split}
	|E_{8}| &=12s^{2}\int_0^T \int_{\mathbb{T}_L}|\partial_{x}\phi\partial_{xx}\phi\partial_{xt}w\partial_{xx}w|
	\\
	& \leqslant \frac{12\alpha_{2}}{2}s\int_0^T \int_{\mathbb{T}_L}|\partial_{xx}\phi(\partial_{xt}w)^{2}| 
	+\frac{12}{2\alpha_{2}}s^{3}\int_0^T \int_{\mathbb{T}_L}|(\partial_{x}\phi)^{2}\partial_{xx}\phi(\partial_{xx}w)^{2}|
	\\
	& \leq \frac{12\alpha_{2}}{14} |E_6| + \frac{12}{2\alpha_{2} |12- 6 \zeta| } |E_3| ,
	\end{split}
	\end{equation}
	We then choose $\zeta,$ such that 
	\begin{equation}
	\mbox{max}\{-8+6\zeta,-66-36\zeta,-12+6\zeta,-3-6\zeta\}<0, 
	\end{equation}
	which imposes $\zeta \in (-1/2, 4/3)$, and such that there exist $\alpha_1>0$ and $\alpha_2>0$ such that
	\begin{equation}
	\label{zetachp3}
	\mbox{max} 
	\left\{
	\frac{(32+12\zeta)}{2\alpha_{1}|66 + 36 \zeta|}, 
	\frac{(32+12\zeta)\alpha_{1}}{2 |3 + 6 \zeta|}, 	
	\frac{12\alpha_{2}}{14} , 
	\frac{12}{2\alpha_{2} |12- 6 \zeta| }
	\right\}
	< 1.
	\end{equation}
	This can be done provided $\zeta \in (-1/2, 4/3)$ satisfies
	\begin{equation*}
	\frac{8+3\zeta}{33+18\zeta} <\frac{3+6\zeta}{16+6\zeta},
	\quad \hbox{and} \quad 
	\frac{1}{2-\zeta}<\frac{{7}}{6}.
	\end{equation*}
	These conditions can be easily satisfied by taking
	\begin{equation}\label{fixzeta}
	\zeta = 1.
	\end{equation}
	At this point in view of the choice \eqref{fixzeta}, we fix $\alpha_{1}$ and $\alpha_{2}$ such that they satisfy \eqref{zetachp3}.
	\\ 
	Hence from \eqref{estsum} we get that there exist positive constants $K_{1},$ $K_{2},$ $K_{3},$ $K_{4},$ $K_{5}$ and $K_{6}$ such that
	\begin{equation}\label{ultest}
	\begin{split}
	\int_0^T \int_{\mathbb{T}_L}P_{1}wP_{2}w
	& \geq -K_{1}s^{7}\int_0^T \int_{\mathbb{T}_L}(\partial_{x}\phi)^{6}\partial_{xx}\phi w^{2}-K_{2}s^{5}\int_0^T \int_{\mathbb{T}_L}(\partial_{x}\phi)^{4}\partial_{xx}\phi(\partial_{x}w)^{2}\\
	&-K_{3}s^{3}\int_0^T \int_{\mathbb{T}_L}(\partial_{x}\phi)^{2}\partial_{xx}\phi(\partial_{xx}w)^{2}
	-K_{4}s^{3}\int_0^T \int_{\mathbb{T}_L}(\partial_{x}\phi)^{2}\partial_{xx}\phi(\partial_{t}w)^{2}
	\\
	&
	-K_{5}s\int_0^T \int_{\mathbb{T}_L}\partial_{xx}\phi(\partial_{xxx}w)^{2}
	-K_{6}s\int_0^T \int_{\mathbb{T}_L}\partial_{xx}\phi(\partial_{xt}w)^{2}+L.O.T.
	\end{split}
	\end{equation}
	Hence in view of \eqref{prilies} and \eqref{Positivity-Weights} one obtains that 
	\begin{equation}\label{postvty}
	\begin{split}
	&\int_0^T \int_{\mathbb{T}_L}P_{1}wP_{2}w \geqslant c\left(
	s^{7}\lambda^{8}\int_0^T \int_{\mathbb{T}_L}{\xi}^{7} w^{2}+s^{5}\lambda^{6}\int_0^T \int_{\mathbb{T}_L}{\xi}^{5} ({\partial_{x}w})^{2}
	\right.\\
	&+s^{3}\lambda^{4}\int_0^T \int_{\mathbb{T}_L}{\xi}^{3} ({\partial_{xx}w})^{2}
	+s^{3}\lambda^{4}\int_0^T \int_{\mathbb{T}_L}{\xi}^{3} ({\partial_{t}w})^{2}
	+s\lambda^{2}\int_0^T \int_{\mathbb{T}_L}{\xi} ({\partial_{xxx}w})^{2}
	\\
	&\left.+s\lambda^{2}\int_0^T \int_{\mathbb{T}_L}{\xi} ({\partial_{xt}w})^{2} \right)
	-C\left(s^{7}\lambda^{8}\iint\limits_{\omega^{2}_{T}}{\xi}^{7} w^{2}
	+s^{5}\lambda^{6}\iint\limits_{\omega^{2}_{T}}{\xi}^{5} ({\partial_{x}w})^{2}\right. \\
	&+s^{3}\lambda^{4}\iint\limits_{\omega^{2}_{T}}{\xi}^{3} ({\partial_{xx}w})^{2}
	+s\lambda^{2}\iint\limits_{\omega^{2}_{T}}{\xi} ({\partial_{xxx}w})^{2}
	+s^{3}\lambda^{4}\iint\limits_{\omega^{2}_{T}}{\xi}^{3} ({\partial_{t}w})^{2}\\
	&\left. +s\lambda^{2}\iint\limits_{\omega^{2}_{T}}{\xi} ({\partial_{xt}w})^{2}\right), 
	\end{split}
	\end{equation}
	where 
	\begin{equation}
	\omega^2_T = (\mathbb{T}_L \setminus [- \frac{L}{2},d+\frac{L}{2}]) \times (0,T).
	\end{equation}
	%
	% 	Using the definition of $\mathit{R}w$ (see \eqref{p12r}) and the inequalities \eqref{prilies} one can check that 
	% 	\begin{equation}\label{estR}
	% 	\begin{split}
	% 	\displaystyle
	% 	\int_0^T \int_{\mathbb{T}_L}|\mathit{R}w|^{2}\leqslant& C(s^{6}\lambda^{8}\int_0^T \int_{\mathbb{T}_L}{\xi}^{7}w^{2}+s^{4}\lambda^{6}\int_0^T \int_{\mathbb{T}_L}{\xi}^{5}(\partial_{x}w)^{2}+s^{2}\lambda^{4}\int_0^T \int_{\mathbb{T}_L}{\xi}^{3}(\partial_{xx}w)^{2}\\
	% 	&+s^{2}\lambda^{4}\int_0^T \int_{\mathbb{T}_L}{\xi}^{3}(\partial_{t}w)^{2}).
	% 	\end{split}
	% 	\end{equation}
	Now in view of \eqref{trieq} and \eqref{R-is-LOT},  \eqref{postvty}  furnishes that for large enough values of the parameter $s$ and $\lambda$ the following holds
	\begin{multline}\label{precarl}
	s^{7}\lambda^{8}\int_0^T \int_{\mathbb{T}_L}{\xi}^{7} w^{2}+s^{5}\lambda^{6}\int_0^T \int_{\mathbb{T}_L}{\xi}^{5} ({\partial_{x}w})^{2}+s^{3}\lambda^{4}\int_0^T \int_{\mathbb{T}_L}{\xi}^{3} ({\partial_{xx}w})^{2}\\
	+s^{3}\lambda^{4}\int_0^T \int_{\mathbb{T}_L}{\xi}^{3} ({\partial_{t}w})^{2}
	+s\lambda^{2}\int_0^T \int_{\mathbb{T}_L}{\xi} ({\partial_{xxx}w})^{2}
	+s\lambda^{2}\int_0^T \int_{\mathbb{T}_L}{\xi} ({\partial_{xt}w})^{2}
	\\
	\leqslant C\left(\int_0^T \int_{\mathbb{T}_L}|{f_{\psi}}|^{2}e^{-2s\phi}
	+s^{7}\lambda^{8}\iint\limits_{\omega^{2}_{T}}{\xi}^{7} w^{2}
	+s^{5}\lambda^{6}\iint\limits_{\omega^{2}_{T}}{\xi}^{5} ({\partial_{x}w})^{2}
	\right. \\
	\left.
	+s^{3}\lambda^{4}\iint\limits_{\omega^{2}_{T}}{\xi}^{3} ({\partial_{xx}w})^{2}
	+s\lambda^{2}\iint\limits_{\omega^{2}_{T}}{\xi} ({\partial_{xxx}w})^{2}
	+s^{3}\lambda^{4}\iint\limits_{\omega^{2}_{T}}{\xi}^{3} ({\partial_{t}w})^{2}
	+s\lambda^{2}\iint\limits_{\omega^{2}_{T}}{\xi} ({\partial_{xt}w})^{2}
	\right).
	\end{multline}
	Now, our goal is to estimate $\partial_{tt} w$, $\partial_{txx} w$ and $\partial_{xxxx} w$. 
	In order to do that, we set 
	\begin{equation}\label{tau}
	\tau=\frac{1}{\sqrt{s\xi}}w.
	\end{equation}
	Using \eqref{eqw}, let us observe that (since $e^{-s\phi}$ vanishes at time $T$) the new unknown $\tau$ solves the following set of equations
	\begin{equation}\label{adjbeam*}
	\left\{ \begin{array}{ll}
	\displaystyle
	{\partial_{tt}\tau} + {\partial_{txx}\tau} + {\partial_{xxxx}\tau}=\frac{1}{\sqrt{s\xi}}f_{\psi}e^{-s\phi}+(\mathcal{F}_{1}+\mathcal{F}_{2}+\mathcal{F}_{3}) -\mathcal{F}_{4}\,\,& \mbox{in}\,\, \mathbb{T}_{L}\times(0,T),\\
	\displaystyle\tau(.,T)=0,\quad{\partial_{t}\tau}(.,T)=0\,\,&\mbox{in}\,\, \mathbb{T}_{L},
	\end{array}\right.
	\end{equation}
	where 
	\begin{equation}\nonumber
	\begin{array}{llll}
	\mathcal{F}_1 &= \partial_{tt}\tau - \frac{1}{\sqrt{s\xi}}{\partial_{tt}w}, 
	\qquad 
	\mathcal{F}_2  &=\partial_{txx}\tau - \frac{1}{\sqrt{s\xi}} {\partial_{txx}w}, 
	\qquad
	\mathcal{F}_3  &=\partial_{xxxx}\tau - \frac{1}{\sqrt{s\xi}}{\partial_{xxxx}w},\\
	&=\left[\partial_{tt},\frac{1}{\sqrt{s\xi}}\right]w, & =\left[\partial_{txx},\frac{1}{\sqrt{s\xi}}\right]w, & =\left[\partial_{xxxx},\frac{1}{\sqrt{s\xi}}\right]w.\\
	\end{array}
	\end{equation}
	and $\mathcal{F}_4$ is given by 
	\begin{multline*}
	\sqrt{s \xi} \mathcal{F}_4 = R w 
	+ s^{4}(\partial_{x}\phi)^{4}w
	+6 s^{2}(\partial_{x}\phi)^{2}\partial_{xx}w
	+ 2 s\partial_{x}\phi\partial_{xt}w
	+ 4s^{3}(\partial_{x}\phi)^{3}\partial_{x}w
	\\
	+4 s\partial_{x}\phi\partial_{xxx}w 
	+  s^{2}(\partial_{x}\phi)^{2}\partial_{t}w
	+6(1+\zeta)s^{3}(\partial_{x}\phi)^{2}\partial_{xx}\phi w.
	\end{multline*}
	It is then easy to check that 
	\begin{multline*}
	\int_0^T \int_{\mathbb{T}_L}\left(|\mathcal{F}_{1}|^{2}+ |\mathcal{F}_{2}|^{2}+ |\mathcal{F}_{3}|^{2}+ |\mathcal{F}_{4}|^{2} \right)
	\\
	\leq 
	C
	\left( s^{7}\lambda^{8}\int_0^T \int_{\mathbb{T}_L}{\xi}^{7} w^{2}+s^{5}\lambda^{6}\int_0^T \int_{\mathbb{T}_L}{\xi}^{5} ({\partial_{x}w})^{2}+s^{3}\lambda^{4}\int_0^T \int_{\mathbb{T}_L}{\xi}^{3} ({\partial_{xx}w})^{2}\right.
	\\
	\left. 
	+s^{3}\lambda^{4}\int_0^T \int_{\mathbb{T}_L}{\xi}^{3} ({\partial_{t}w})^{2}
	+s\lambda^{2}\int_0^T \int_{\mathbb{T}_L}{\xi} ({\partial_{xxx}w})^{2}
	+s\lambda^{2}\int_0^T \int_{\mathbb{T}_L}{\xi} ({\partial_{xt}w})^{2}\right).
	\end{multline*}
	Hence the maximal parabolic regularity (we refer to Lemma \ref{lemmaexistence} for details) result for the system \eqref{adjbeam*} furnishes the following
	\begin{equation}\label{maxpra}
	\tau \in L^{2}(0,T;H^{4}(\mathbb{T}_{L})) \cap H^{2}(0,T; L^{2}(\mathbb{T}_{L})).
	\end{equation}
	Besides one has the following inequality
	\begin{multline}\label{energyin}
	\|\tau \|^{2}_{L^{2}(0,T;H^{4}(\mathbb{T}_{L})) \cap H^{2}(0,T; L^{2}(\mathbb{T}_{L}))}
	\leqslant C(\|{f_{\psi}}e^{-s\phi}\|^{2}_{L^{2}(\mathbb{T}_L \times (0,T))}+\|\mathcal{F}_{1}\|^{2}_{L^{2}(\mathbb{T}_L \times (0,T))}
	\\+\|\mathcal{F}_{2}\|^{2}_{L^{2}(\mathbb{T}_L \times (0,T))}
	+\|\mathcal{F}_{3}\|^{2}_{L^{2}(\mathbb{T}_L \times (0,T))}+\|\mathcal{F}_{4}\|^{2}_{L^{2}(\mathbb{T}_L \times (0,T))}).
	\end{multline}
	This then yields the following estimate:
	\begin{multline}\label{hoddr}
	\frac{1}{s}\int_0^T \int_{\mathbb{T}_L}\frac{1}{{\xi}}(|{\partial_{tt}w}|^{2}+|\partial_{txx}w|^{2}+|{\partial_{xxxx}w}|^{2})e^{-2s\phi}
	\leqslant C(\|{f_{\psi}}e^{-s\phi}\|^{2}_{L^{2}(\mathbb{T}_L \times (0,T))}+\|\mathcal{F}_{1}\|^{2}_{L^{2}(\mathbb{T}_L \times (0,T))}
	\\+\|\mathcal{F}_{2}\|^{2}_{L^{2}(\mathbb{T}_L \times (0,T))}
	+\|\mathcal{F}_{3}\|^{2}_{L^{2}(\mathbb{T}_L \times (0,T))}+\|\mathcal{F}_{4}\|^{2}_{L^{2}(\mathbb{T}_L \times (0,T))}).
	\end{multline}

	Combining the inequalities \eqref{precarl} and \eqref{hoddr} one obtains the following
	\begin{multline}\label{carlmobs}
	s^{7}\lambda^{8}\int_0^T \int_{\mathbb{T}_L}{\xi}^{7} w^{2}+s^{5}\lambda^{6}\int_0^T \int_{\mathbb{T}_L}{\xi}^{5} ({\partial_{x}w})^{2}+s^{3}\lambda^{4}\int_0^T \int_{\mathbb{T}_L}{\xi}^{3} ({\partial_{xx}w})^{2}\\
	+s\lambda^{2}\int_0^T \int_{\mathbb{T}_L}{\xi} ({\partial_{xxx}w})^{2}
	+s^{3}\lambda^{4}\int_0^T \int_{\mathbb{T}_L}{\xi}^{3} ({\partial_{t}w})^{2}
	+s\lambda^{2}\int_0^T \int_{\mathbb{T}_L}{\xi} ({\partial_{xt}w})^{2}
	\\
	+\frac{1}{s}\int_0^T \int_{\mathbb{T}_L}\frac{1}{{\xi}}(|{\partial_{tt}w}|^{2}+|\partial_{txx}w|^{2}+|{\partial_{xxxx}w}|^{2})e^{-2s\phi}
	\\
	\leqslant C\left(\int_0^T \int_{\mathbb{T}_L}|{f_{\psi}}|^{2}e^{-2s\phi}
	+s^{7}\lambda^{8}\iint\limits_{\omega^{2}_{T}}{\xi}^{7} w^{2}
	+s^{5}\lambda^{6}\iint\limits_{\omega^{2}_{T}}{\xi}^{5} ({\partial_{x}w})^{2}+s^{3}\lambda^{4}\iint\limits_{\omega^{2}_{T}}{\xi}^{3} ({\partial_{xx}w})^{2}
	\right. \\ 
	\left.
	+s^{3}\lambda^{4}\iint\limits_{\omega^{2}_{T}}{\xi}^{3} ({\partial_{t}w})^{2}
	+s\lambda^{2}\iint\limits_{\omega^{2}_{T}}{\xi} ({\partial_{xxx}w})^{2}
	+s\lambda^{2}\iint\limits_{\omega^{2}_{T}}{\xi} ({\partial_{xt}w})^{2}
	\right).
	\end{multline}
	
	Now, we need to suitably absorb the third to seventh observability terms appearing in the R.H.S of \eqref{carlmobs}. This is rather standard and such arguments can be found for instance in \cite[p. 461]{farnan} and \cite[p. 565]{ervbad}. We absorb it in a reverse way, starting from the last terms. 
	
	We introduce a smooth cut-off function $\Upsilon_2$ such that 
	\begin{multline}\label{definitionUpsilon2}
	\Upsilon_{2}\in C^{\infty}_{c}(\mathbb{T}_L; [0,1]),
	\quad
	\Upsilon_{2}(x)=1\,\mbox{in}\, \omega^{2},
	\quad
	\Upsilon_2(x) = 0\, \mbox{for} \, x \notin \omega^3, 
	\\
	\hbox{ where } \omega^2 = \mathbb{T}_L \setminus [-\frac{L}{2},d+\frac{L}{2}], 
	\hbox{ and } \omega^3 = \mathbb{T}_L \setminus [-\frac{L}{4},d+\frac{L}{4}].
	\end{multline}
	In the following, we shall also use the notation $\omega^3_T = \omega^3 \times (0,T)$.
	
	Using Young's inequality, we have, for all $\varepsilon >0$,
	\begin{align*}%\label{obsinoi5}
	& \displaystyle s\lambda^{2}\iint\limits_{\omega^2_T}
	\xi(\partial_{xt}w)^{2}
	\displaystyle \leqslant s\lambda^{2}\iint\limits_{\omega^3_T} \Upsilon_{2}\xi(\partial_{xt}w)^2 =  s\lambda^{2}\iint\limits_{\omega^3_T} \Upsilon_{2}\xi \partial_{xx}w \partial_{tt} w + L.O.T.
	\\
	& \leq \frac{\varepsilon}{2s} \iint\limits_{\omega^3_T} \Upsilon_{2}\xi^{-1}(\partial_{tt} w)^2
	+ \frac{s^3 \lambda^4}{2 \varepsilon} \iint\limits_{\omega^3_T} \Upsilon_{2}\xi^{3}(\partial_{xx} w)^2 + L.O.T.
	\\
	& \leq \frac{\varepsilon}{2s} \iint\limits_{\omega^3_T} \xi^{-1}(\partial_{tt} w)^2
	+ \frac{s^3 \lambda^4}{2 \varepsilon} \iint\limits_{\omega^3_T} \xi^{3}(\partial_{xx} w)^2 + L.O.T.
	\end{align*}
	Similarly, we get 
	\begin{align*}
	& \displaystyle s\lambda^{2}\iint\limits_{\omega^2_T} \xi(\partial_{xxx}w)^2
	\leqslant s\lambda^{2}\iint\limits_{\omega^3_T} \Upsilon_{2}\xi(\partial_{xxx}w)^2 
	= - s \lambda^2 \iint\limits_{\omega^3_T} \Upsilon_{2}\xi\partial_{xxxx} w \partial_{xx}w + L.O.T.
	\\
	& \leq \frac{\varepsilon}{2s} \iint\limits_{\omega^3_T} \xi^{-1} (\partial_{xxxx} w)^2 + \frac{s^3 \lambda^4}{2 \varepsilon} \iint\limits_{\omega^3_T} \xi^{3} (\partial_{xx} w)^2 + L.O.T.
	\end{align*}
	We then choose $\varepsilon >0$ small enough so that $C \varepsilon <1$, where $C$ is the constant in \eqref{carlmobs}, and we plug these two estimates in \eqref{carlmobs}. We obtain that there exists a constant $C>0$ such that for all $s$ and $\lambda$ large enough, 
	\begin{multline}\label{carlmobs-2}
	s^{7}\lambda^{8}\int_0^T \int_{\mathbb{T}_L}{\xi}^{7} w^{2}+s^{5}\lambda^{6}\int_0^T \int_{\mathbb{T}_L}{\xi}^{5} ({\partial_{x}w})^{2}+s^{3}\lambda^{4}\int_0^T \int_{\mathbb{T}_L}{\xi}^{3} ({\partial_{xx}w})^{2}\\
	+s\lambda^{2}\int_0^T \int_{\mathbb{T}_L}{\xi} ({\partial_{xxx}w})^{2}
	+s^{3}\lambda^{4}\int_0^T \int_{\mathbb{T}_L}{\xi}^{3} ({\partial_{t}w})^{2}
	+s\lambda^{2}\int_0^T \int_{\mathbb{T}_L}{\xi} ({\partial_{xt}w})^{2}
	\\
	+\frac{1}{s}\int_0^T \int_{\mathbb{T}_L}\frac{1}{{\xi}}(|{\partial_{tt}w}|^{2}+|\partial_{txx}w|^{2}+|{\partial_{xxxx}w}|^{2})e^{-2s\phi}
	\leqslant C\left(\int_0^T \int_{\mathbb{T}_L}|{f_{\psi}}|^{2}e^{-2s\phi}
	\right.
	\\
	\left.
	+s^{7}\lambda^{8}\iint\limits_{\omega^{3}_{T}}{\xi}^{7} w^{2}
	+s^{5}\lambda^{6}\iint\limits_{\omega^{3}_{T}}{\xi}^{5} ({\partial_{x}w})^{2}+s^{3}\lambda^{4}\iint\limits_{\omega^{3}_{T}}{\xi}^{3} ({\partial_{xx}w})^{2}
	+s^{3}\lambda^{4}\iint\limits_{\omega^{3}_{T}}{\xi}^{3} ({\partial_{t}w})^{2}
	+L.O.T. \right).
	\end{multline}
	
	We now introduce a smooth cut-off function $\Upsilon_3$ such that 
	$$
	\Upsilon_{3}\in C^{\infty}_{c}(\mathbb{T}_L; [0,1]),
	\quad
	\Upsilon_{3}(x)=1\,\mbox{in}\, \omega^{3},
	\quad
	\Upsilon_3(x) = 0\, \mbox{in}\, [0,d], 
	$$
	and we use the notation $\omega_1 = \mathbb{T}_L\setminus[0,d]$, and $\omega_1^T = \omega_1 \times (0,T)$.
	\\
	Now, as before we can write, for $\varepsilon_1 >0$ to be fixed later,
	\begin{align*}
	s^{3}\lambda^{4}\iint\limits_{\omega^{3}_{T}}{\xi}^{3} ({\partial_{t}w})^{2}
	& \leq 
	s^{3}\lambda^{4}\iint\limits_{\omega^{1}_{T}}\Upsilon_3 {\xi}^{3} ({\partial_{t}w})^{2}
	= - s^3 \lambda^4 \iint\limits_{\omega^{1}_{T}}\Upsilon_3 {\xi}^{3} {\partial_{tt}w} w + L.O.T
	\\
	& \leq \frac{\varepsilon_1}{2s} \iint\limits_{\omega^{1}_{T}} \xi^{-1} (\partial_{tt} w)^2 
	+ \frac{s^7 \lambda^8}{2 \varepsilon_1} \iint\limits_{\omega^{1}_{T}}{\xi}^{7} w^{2} +L.O.T, 
	\end{align*}
	and 
	\begin{align*}
	s^{3}\lambda^{4}\iint\limits_{\omega^{3}_{T}}{\xi}^{3} ({\partial_{xx}w})^{2}
	& \leq 
	s^{3}\lambda^{4}\iint\limits_{\omega^{1}_{T}}\Upsilon_3 {\xi}^{3} ({\partial_{xx}w})^{2}
	=  s^3 \lambda^4 \iint\limits_{\omega^{1}_{T}}\Upsilon_3 {\xi}^{3} {\partial_{xxxx}w} w + L.O.T
	\\
	& \leq \frac{\varepsilon_1}{2s} \iint\limits_{\omega^{1}_{T}} \xi^{-1} (\partial_{xxxx} w)^2 + \frac{s^7 \lambda^8}{2 \varepsilon_1} \iint\limits_{\omega^{1}_{T}}{\xi}^{7} w^{2} +L.O.T.
	\end{align*}
	Similarly, 
	\begin{align*}
	&
	s^{5}\lambda^{6}\iint\limits_{\omega^{3}_{T}}{\xi}^{5} ({\partial_{x}w})^{2}
	\leq 
	s^{5}\lambda^{6}\iint\limits_{\omega^{1}_{T}}\Upsilon_3 {\xi}^{5} ({\partial_{x}w})^{2}
	= - s^5 \lambda^6 \iint\limits_{\omega^{1}_{T}}\Upsilon_3 {\xi}^{5} {\partial_{xx}w} w + L.O.T.
	\\
	& \leq \frac{\varepsilon_1 s^3 \lambda^4 }{2} \iint\limits_{\omega^{1}_{T}} \xi^{3} (\partial_{xx} w)^2 + \frac{s^7 \lambda^8}{2 \varepsilon_1} \iint\limits_{\omega^{1}_{T}}{\xi}^{7} w^{2} +L.O.T.
	\end{align*}
	Choosing now	$\varepsilon_1>0$ small enough so that $C \varepsilon_1 < 1$ where $C$ is the constant in \eqref{carlmobs-2}, we deduce the following inequality from \eqref{carlmobs-2}: for all $s$ and $\lambda$ large enough,
	\begin{multline}\label{carlmobs-2-bis}
	s^{7}\lambda^{8}\int_0^T \int_{\mathbb{T}_L}{\xi}^{7} w^{2}+s^{5}\lambda^{6}\int_0^T \int_{\mathbb{T}_L}{\xi}^{5} ({\partial_{x}w})^{2}+s^{3}\lambda^{4}\int_0^T \int_{\mathbb{T}_L}{\xi}^{3} ({\partial_{xx}w})^{2}\\
	+s\lambda^{2}\int_0^T \int_{\mathbb{T}_L}{\xi} ({\partial_{xxx}w})^{2}
	+s^{3}\lambda^{4}\int_0^T \int_{\mathbb{T}_L}{\xi}^{3} ({\partial_{t}w})^{2}
	+s\lambda^{2}\int_0^T \int_{\mathbb{T}_L}{\xi} ({\partial_{xt}w})^{2}
	\\
	+\frac{1}{s}\int_0^T \int_{\mathbb{T}_L}\frac{1}{{\xi}}(|{\partial_{tt}w}|^{2}+|\partial_{txx}w|^{2}+|{\partial_{xxxx}w}|^{2})e^{-2s\phi}
	\\
	\leqslant C\left(\int_0^T \int_{\mathbb{T}_L}|{f_{\psi}}|^{2}e^{-2s\phi}
	%	 \right.
	%	 \\
	%	 \left.
	+s^{7}\lambda^{8}\iint\limits_{\omega^{1}_{T}}{\xi}^{7} w^{2}
	+L.O.T. \right).
	\end{multline}
	Now, the lower order terms $L.O.T$ can be absorbed by taking $s$ and $\lambda$ large enough, so that from \eqref{carlmobs-2-bis}, we obtain that  
	for all $s$ and $\lambda$ large enough,
	\begin{multline}\label{carlmobs**}
	s^{7}\lambda^{8}\int_0^T \int_{\mathbb{T}_L}{\xi}^{7} w^{2}+s^{5}\lambda^{6}\int_0^T \int_{\mathbb{T}_L}{\xi}^{5} ({\partial_{x}w})^{2}+s^{3}\lambda^{4}\int_0^T \int_{\mathbb{T}_L}{\xi}^{3} ({\partial_{xx}w})^{2}\\
	+s\lambda^{2}\int_0^T \int_{\mathbb{T}_L}{\xi} ({\partial_{xxx}w})^{2}
	+s^{3}\lambda^{4}\int_0^T \int_{\mathbb{T}_L}{\xi}^{3} ({\partial_{t}w})^{2}
	+s\lambda^{2}\int_0^T \int_{\mathbb{T}_L}{\xi} ({\partial_{xt}w})^{2}
	\\
	+\frac{1}{s}\int_0^T \int_{\mathbb{T}_L}\frac{1}{{\xi}}(|{\partial_{tt}w}|^{2}+|\partial_{txx}w|^{2}+|{\partial_{xxxx}w}|^{2})e^{-2s\phi}
	\\
	\leqslant C\left(\int_0^T \int_{\mathbb{T}_L}|{f_{\psi}}|^{2}e^{-2s\phi}
	+s^{7}\lambda^{8}\iint\limits_{\omega^{1}_{T}}{\xi}^{7} w^{2}
	\right).
	\end{multline}
	To obtain \eqref{crlestbmthm} from \eqref{carlmobs**} we just need to recall that $w = e^{-s\phi}\psi$, or equivalently that $\psi = w e^{s \phi}$. This argument is very standard and is left to the reader.
\end{proof}
The following corollary corresponds to a Carleman estimate with potential is a direct consequence of Theorem \ref{Carlbeamthm}.
\begin{corollary}\label{corpotential}
	Let the potential $a\in L^{\infty}({\mathbb{T}^{d}_{L}}\times(0,T)).$ There exist a constant $\lambda_{1}>1,$ independent of $\|a\|_{L^{\infty}(\mathbb{T}^{d}_{L}\times(0,T))}$and constants $C=C(\|a\|_{L^{\infty}(\mathbb{T}^{d}_{L}\times(0,T))})>0$ and $s_{1}=s_{1}(\|a\|_{L^{\infty}(\mathbb{T}^{d}_{L}\times(0,T))})> 1$ such that for all smooth functions $\psi$ on ${\mathbb{T}^{d}_{L}}\times[0,T],$ for all $s\geqslant s_{1}$ and $\lambda\geqslant\lambda_{1},$
	\begin{align}\label{crlestbmcor}
	& \displaystyle s^{7}\lambda^{8}\int_0^T \int_{{\mathbb{T}^{d}_{L}}} {\xi}^{7}|{\psi}|^{2}e^{-2s\phi}+s^{5}\lambda^{6}\int_0^T \int_{{\mathbb{T}^{d}_{L}}} {\xi}^{5}|{\partial_{x}{\psi}}|^{2}e^{-2s\phi}\notag
	%	\\
	% 	&\displaystyle
	\\
	&\displaystyle+s^{3}\lambda^{4}\int_0^T \int_{{\mathbb{T}^{d}_{L}}} {\xi}^{3} ( |{\partial_{xx}{\psi}}|^{2}+ |{\partial_{t}{\psi}}|^{2}) e^{-2s\phi}+s\lambda^{2}\int_0^T \int_{{\mathbb{T}^{d}_{L}}} {\xi}( |{\partial_{tx}{\psi}}|^{2}+|{\partial_{xxx}{\psi}}|^{2})e^{-2s\phi}\notag \\
	%	\\
	%	&
	&+
	\displaystyle\frac{1}{s}\int_0^T \int_{{\mathbb{T}^{d}_{L}}} \frac{1}{{\xi}}(|{\partial_{tt}{\psi}}|^{2}+|\partial_{txx}\psi|^{2}+|{\partial_{xxxx}{\psi}}|^{2})e^{-2s\phi}
	\\
	&
	\displaystyle\leqslant  
	C\int_0^T \int_{{\mathbb{T}^{d}_{L}}} |(\partial_{tt}+\partial_{txx}+\partial_{xxxx}+a)\psi|^{2}e^{-2s\phi}
	+Cs^{7}\lambda^{8}\int_0^T \int_{\omega} {\xi}^{7}|{\psi}|^{2}e^{-2s\phi},\notag
	\end{align}
	where the notation $\omega$ was introduced in \eqref{controlzone}.
\end{corollary}
\begin{proof}
	We observe that
	\begin{align*}
	\int_0^T \int_{{\mathbb{T}^{d}_{L}}} |(\partial_{tt}+\partial_{txx}+\partial_{xxxx}+a)\psi|^{2}e^{-2s\phi}\leqslant C&\left( \int_0^T \int_{{\mathbb{T}^{d}_{L}}} |(\partial_{tt}+\partial_{txx}+\partial_{xxxx})\psi|^{2}e^{-2s\phi}\right.\\
	&\left.+\|a\|^{2}_{L^{\infty}(\mathbb{T}^{d}_{L}\times(0,T))}\int_0^T \int_{{\mathbb{T}^{d}_{L}}}|\psi|^{2}e^{-2s\phi}\right)
	\end{align*}
	for some positive constant $C$ independent of $a.$ Now one can readily use the inequality \eqref{crlestbmthm} and choose $s_{1}=s_{1}(\|a\|_{L^{\infty}(\mathbb{T}^{d}_{L}\times(0,T))})$ large enough such that if $s\geqslant s_{1},$ the term $\displaystyle\|a\|^{2}_{L^{\infty}(\mathbb{T}^{d}_{L}\times(0,T))}\int_0^T \int_{{\mathbb{T}^{d}_{L}}}|\psi|^{2}e^{-2s\phi}$ can be absorbed by the term $\displaystyle s^{7}\lambda^{8}\int_0^T \int_{{\mathbb{T}^{d}_{L}}} {\xi}^{7}|{\psi}|^{2}e^{-2s\phi}$ appearing in the left hand side on \eqref{crlestbmthm}. Consequently one obtains \eqref{crlestbmcor}.
	\end{proof}
	The Corollary \ref{corpotential} will be used in next section to prove Theorem \ref{Centraltheoremnullcont}.
	%%%%%%%%%%%%%%%%%%%%%%%%%%%%%%%%%%%%%%%%%%%%%%%%%%%%%%%%%%%%%%%%%%%%%%%%%%%%%%%%%%%%%%%%%%%%%%%%%%%%%%%%%%%%%%%%%%%%%%%%%%%%%%%%%%%%%%%%%%%%%%%%%%%%%%%%%%%%%%%%%%%%%%%%%%%%%%%%%%%%%%%%%%%%%%%%%%%%%%%%%%%%%%%%%%%%%%%%%%%%%%%%%%%%%%%%%%%%%%%%%%%%%%%%%%%%%%%%%%%%%%%%%%%%%%%%%%%%%%%%%%%%%%%%%%%%%%%%%%%%%%%%%%%%%%%%%%%%%%%%%%%%%%%%%%%%%%%%%%%%%%%%%%%%%%%%%%%%%%%%%%%%%%%%%%%%%%%%%%%%%%%%%%%%%%%%%%%%%%%%%%%%%%%%%%%%%%%%%%%%%%%%%%%%%%%%%%%%%%%%%%%%%%%%%%%%%%%%%%%%%%%%%%%%%%%%%%%%%%%%%%%%%%%%%%%%%%%%%%%%%%%%%%%%%%
	\section{Proof of Theorem \ref{Centraltheoremnullcont} }\label{Nullcontrollability}
	This section is dedicated to the proof of the Theorem \ref{Centraltheoremnullcont}. The proof will be based on a duality approach. 
	\begin{proof}[Proof of Theorem \ref{Centraltheoremnullcont}] 
    	We fix the parameters $s$ and $\lambda$ in the Carleman inequality \eqref{crlestbmcor}.\\
    	 The idea is to write the control problem \eqref{dampedbeam} as the sum of a problem with inhomogeneous initial conditions and a different control problem with homogeneous initial conditions. We will only control the homogeneous initial value problem and show that the same control is sufficient to drive the state $\beta,$ the solution to \eqref{dampedbeam} to zero.\\
    	 In that direction let us first introduce a cut-off function $\theta_{1}(t)$ in time as follows:
    	 \begin{equation}\label{cutofftheta}
    	 \begin{array}{l}
    	 \theta_{1}(t)\in C^{\infty}\,\,\mbox{in a neighborhood of}\,\,(0,T)\,\, \mbox{such that},\\
    	 \theta_{1}(t)=0\,\,\mbox{in a neighborhood of}\,\, \{T\},\,\,\mbox{and}\,\,\theta_{1}(t)=1\,\,\mbox{in a neighborhood of}\,\,\{0\}.
    	 \end{array}
    	 \end{equation}
    	 We decompose $\beta,$ the solution of \eqref{dampedbeam} as follows:
    	 \begin{equation}\label{decomposebeta}
    	 \begin{array}{l}
    	 \beta(x,t)=\theta_{1}(t)q(x,t)+g(x,t),
    	 \end{array}
    	 \end{equation} 
    	 where $q$ solves 
    	 \begin{equation}\label{dampedbeamq}
    	 \left\{ \begin{array}{ll}
    	 \displaystyle\partial_{tt}q-\partial_{txx}q+\partial_{xxxx}q+aq=0 \,& \mbox{in}\, {\mathbb{T}^{d}_{L}}\times(0,T),
    	 \vspace{1.mm}\\
    	 \displaystyle q(\cdot,0)=\beta_{0}\quad \mbox{and}\quad \partial_{t} q(\cdot,0)=\beta_{1}\, &\mbox{in}\, {\mathbb{T}^{d}_{L}}
    	 \end{array}\right.
    	 \end{equation}
    	 and $g$ satisfies the following system
    	 \begin{equation}\label{dampedbeamw}
    	 \left\{ \begin{array}{ll}
    	 \displaystyle\partial_{tt}g-\partial_{txx}g+\partial_{xxxx}g+ag=v_{\beta}\chi_{\omega}+f_{\theta_{1},q} \,& \mbox{in}\, {\mathbb{T}^{d}_{L}}\times(0,T),
    	 \vspace{1.mm}\\
    	 \displaystyle g(\cdot,0)=0\quad \mbox{and}\quad \partial_{t} g(\cdot,0)=0\, &\mbox{in}\, {\mathbb{T}^{d}_{L}},
    	 \end{array}\right.
    	 \end{equation}
    	 where 
    	 \begin{equation}\label{fbetaq}
    	 \begin{array}{l}
    	 \displaystyle f_{\theta_{1},q}=-\theta_{1}''(t)q-2\theta_{1}'(t)\partial_{t}q+\theta_{1}'(t)\partial_{xx}q.
    	 \end{array}
    	 \end{equation}
    	 Since $\theta_{1}$ vanishes near $\{T\},$ we observe that the control $v_{\beta}\chi_{\omega}$ which drives $g,$ the solution of \eqref{dampedbeamw} to zero also gives the null controllability of $\beta,$ the solution to \eqref{dampedbeam}. Hence we will focus in constructing $v_{\beta}\chi_{\omega}$ such that $g$ satisfies the following null controllability requirement
    	 \begin{equation}\label{nullcontrolrequirementw}
    	 \begin{array}{l}
    	 (g,\partial_{t}g)(\cdot,T)=(0,0).
    	 \end{array}
    	 \end{equation}
	In that direction we first write the control problem \eqref{dampedbeamw} under a weak form.\\
	We multiply the equation \eqref{dampedbeamw} by smooth functions $\psi$ on $\overline{{\mathbb{T}^{d}_{L}}}\times[0,T].$ We obtain:
	\begin{equation}\label{weakformulation}
	\begin{array}{ll}
	&\displaystyle
	\int_0^T \int_{{\mathbb{T}^{d}_{L}}}g\partial_{tt}\psi-\int\limits_{{\mathbb{T}^{d}_{L}}}\partial_{t}\psi(T)g(T)+\int\limits_{{\mathbb{T}^{d}_{L}}}\psi(T)\partial_{t}g(T)
	\displaystyle +\int_0^T \int_{{\mathbb{T}^{d}_{L}}}g\partial_{txx}\psi-\int\limits_{{\mathbb{T}^{d}_{L}}}g(T)\partial_{xx}\psi(T)\\
	&\displaystyle+\int_0^T \int_{{\mathbb{T}^{d}_{L}}}g\partial_{xxxx}\psi+\int_0^T \int_{{\mathbb{T}^{d}_{L}}}a\psi
	=\int_0^T \int_{\omega}v_{\beta}\psi+\int_0^T \int_{{\mathbb{T}^{d}_{L}}}f_{\theta_{1},q}\psi.
	\end{array}
	\end{equation} 
	In view of \eqref{weakformulation}, the null controllability requirement \eqref{nullcontrolrequirementw} is satisfied if and only if the following holds for all smooth functions $\psi$ on $\overline{{\mathbb{T}^{d}_{L}}}\times[0,T]:$ 
	\begin{equation}\label{weakformulationaftercontrol}
	\begin{array}{ll}
	&\displaystyle\int_0^T \int_{{\mathbb{T}^{d}_{L}}}g\partial_{tt}\psi+\int_0^T \int_{{\mathbb{T}^{d}_{L}}}g\partial_{txx}\psi+\int_0^T \int_{{\mathbb{T}^{d}_{L}}}g\partial_{xxxx}\psi+\int_0^T \int_{{\mathbb{T}^{d}_{L}}}a\psi\\
	&\displaystyle=\int_0^T \int_{\omega}v_{\beta}\psi+\int_0^T \int_{{\mathbb{T}^{d}_{L}}}f_{\theta_{1},q}\psi.
	\end{array}
	\end{equation}
	The trick now is to introduce a functional $J$ whose Euler Lagrange equation coincide with \eqref{weakformulationaftercontrol}: For smooth functions $\psi$ on $\overline{{\mathbb{T}^{d}_{L}}}\times[0,T],$ we define
	\begin{equation}\label{thefunctioanl}
	\begin{array}{ll}
	\displaystyle \displaystyle J(\psi)&\displaystyle=\frac{1}{2}\int_0^T \int_{{\mathbb{T}^{d}_{L}}}|(\partial_{tt}+\partial_{txx}+\partial_{xxxx}+a)\psi|^{2}e^{-2s\phi}+\frac{s^{7}\lambda^{8}}{2}\int_0^T \int_{{\mathbb{T}^{d}_{L}}}\chi_{\omega}^{2}\xi^{7}|\psi|^{2}e^{-2s\phi}\\
	&\displaystyle-\int_0^T \int_{{\mathbb{T}^{d}_{L}}}f_{\theta_{1},q}\psi.
	\end{array}
	\end{equation} 
	But the set of smooth functions on $\overline{\mathbb{T}^{d}_{L}}\times[0,T]$ is not a Banach space. This leads us to define
	$$H_{obs}=\overline{\{\psi\in C^{\infty}({\mathbb{T}^{d}_{L}}\times[0,T])\}}^{\|\cdot\|_{obs}},$$
	where the over line refers to the completion with respect to the Hilbert norm $\|\cdot\|_{obs}$ defined by
	 \begin{equation}\label{Hilbertnorm}
	 \begin{array}{l}
	 \displaystyle\|\psi\|^{2}_{obs}=\int_0^T \int_{{\mathbb{T}^{d}_{L}}}|(\partial_{tt}+\partial_{txx}+\partial_{xxxx}+a)\psi|^{2}e^{-2s\phi}+{s^{7}\lambda^{8}}\int_0^T \int_{{\mathbb{T}^{d}_{L}}}\chi_{\omega}^{2}\xi^{7}|\psi|^{2}e^{-2s\phi}.
	 \end{array}
	 \end{equation} 
	In view of the Carleman estimate \eqref{crlestbmcor}, we conclude that $\|\cdot\|_{obs}$ defines a norm indeed.\\
   Let us show that $J(\psi)$ can be extended as a continuous function on $H_{obs}.$\\ 
  First of all we observe that since  $f_{\theta_{1},q}$ vanishes in a neighborhood of $0$ and $T$ and the parameters $s$ and $\lambda$ are fixed, hence from \eqref{fbetaq} one furnishes
   \begin{equation}\label{boundftheta1q}
   \begin{array}{l}
   \displaystyle\int_0^T \int_{{\mathbb{T}^{d}_{L}}}|f_{\theta_{1},q}|^{2}e^{2s\phi}\leqslant C\|(q,\partial_{t}q)\|^{2}_{L^{2}(0,T;H^{2}({\mathbb{T}^{d}_{L}}))\times L^{2}(0,T;L^{2}({\mathbb{T}^{d}_{L}}))}\leqslant C\|(\beta_{0},\beta_{1})\|^{2}_{H^{3}({\mathbb{T}^{d}_{L}})\times H^{1}({\mathbb{T}^{d}_{L}})},
   \end{array}
   \end{equation}
   for some positive constant $C.$ The second inequality of \eqref{boundftheta1q} follows by using Lemma \ref{existencepotential} (with $v_{\beta}=0$) in view of the initial regularity assumption \eqref{regularityinitial}.\\
	We observe that for a smooth function $\psi$ on $\overline{\mathbb{T}^{d}_{L}}\times[0,T]$ the following holds as a consequence of the Carleman estimate \eqref{crlestbmcor}
	\begin{equation}\label{Holder}
	\begin{array}{l}
	\displaystyle\int_0^T \int_{{\mathbb{T}^{d}_{L}}}f_{\theta_{1},q}\psi\leqslant C\|\psi\|_{obs}\left(\int_0^T \int_{{\mathbb{T}^{d}_{L}}}|f_{\theta_{1},q}|^{2}e^{2s\phi}\right)^{1/2}\leqslant C\|\psi\|_{obs}\|(\beta_{0},\beta_{1})\|_{H^{3}({\mathbb{T}^{d}_{L}})\times H^{1}({\mathbb{T}^{d}_{L}})},
	\end{array}
	\end{equation}
	for some positive constant $C$. The second inequality of \eqref{Holder} follows from \eqref{boundftheta1q}.\\
	In view of \eqref{Holder} the following map
	\begin{equation}\label{contextension}
	\begin{array}{l}
	\displaystyle\psi\longmapsto \int_{{\mathbb{T}^{d}_{L}}}f_{\theta_{1},q}\psi, 
	\end{array}
	\end{equation}
	admits of a continuous extension on the space $H_{obs}.$ This further implies our claim, $i.e$ $J(\psi)$ can be extended as a continuous function on $H_{obs}.$\\
	Now we claim that $J(\psi)$ on $H_{obs}$ is coercive. 
 In view of the definition \eqref{thefunctioanl} of $J(\psi)$ and the inequality \eqref{Holder}, one furnishes the following 
	\begin{equation}\label{coercivitystp1}
	\begin{array}{ll}
	\displaystyle J(\psi)\geqslant &\displaystyle\frac{1}{2}\int_0^T \int_{{\mathbb{T}^{d}_{L}}}|(\partial_{tt}+\partial_{txx}+\partial_{xxxx}+a )\psi|^{2}e^{-2s\phi}+\frac{s^{7}\lambda^{8}}{2}\int_0^T \int_{{\mathbb{T}^{d}_{L}}}\chi_{\omega}^{2}\xi^{7}|\psi|^{2}e^{-2s\phi}\\[3.mm]
	& \displaystyle -C\|\psi\|_{obs}\|(\beta_{0},\beta_{1})\|_{H^{3}({\mathbb{T}^{d}_{L}})\times H^{1}({\mathbb{T}^{d}_{L}})}\\[3.mm]
	&\displaystyle
	\geqslant\frac{1}{2}\int_0^T \int_{{\mathbb{T}^{d}_{L}}}|(\partial_{tt}+\partial_{txx}+\partial_{xxxx}+a )\psi|^{2}e^{-2s\phi}+\frac{s^{7}\lambda^{8}}{2}\int_0^T \int_{{\mathbb{T}^{d}_{L}}}\chi_{\omega}^{2}\xi^{7}|\psi|^{2}e^{-2s\phi}\\[3.mm]
	&\displaystyle -C\left(\frac{\epsilon}{2}\|\psi\|_{obs}+\frac{1}{2\epsilon}\|(\beta_{0},\beta_{1})\|_{H^{3}({\mathbb{T}^{d}_{L}}\times H^{1}({\mathbb{T}^{d}_{L}})}\right)
	\end{array}
	\end{equation} 
	for some positive constant $C$ and a positive parameter $\epsilon.$ Choosing $\epsilon$ to be sufficiently small and making use of the definition \eqref{Hilbertnorm} of $\|\cdot\|_{obs},$ one furnishes from \eqref{coercivitystp1} that $J(\psi)\rightarrow +\infty$ as $\|\psi\|_{obs}\rightarrow +\infty.$ This furnishes the coercivity of $J(\psi)$ on $H_{obs}.$\\
	On the other hand it is easy to verify that $J(\psi)$ is convex.\\
	So far, we have seen that $J(\psi)$ is convex and coercive on $H_{obs}.$ Therefore it admits of a unique minimizer $\psi_{\min}$ on $H_{obs}.$ Let us set 
	\begin{equation}\label{expressionsolutioncontrol}
	\begin{array}{l}
	\displaystyle\widetilde{g}=e^{-2s\phi}(\partial_{tt}+\partial_{txx}+\partial_{xxxx}+a)\psi_{\min},\quad {v}_{\widetilde\beta}=-s^{7}\lambda^{8}\xi^{7}\chi_{\omega}\psi_{\min}e^{-2s\phi}.
	\end{array}
	\end{equation}
	Now we write the Euler Lagrange equation of $J$ at $\psi_{min},$ for all smooth function $\psi$ on $\overline{{\mathbb{T}^{d}_{L}}}\times[0,T]$
	 \begin{equation}\label{EulerLagrange}
	 \begin{array}{l}
	 \displaystyle\int_0^T \int_{{\mathbb{T}^{d}_{L}}}\widetilde{g}(\partial_{tt}+\partial_{txx}+\partial_{xxxx}+a)\psi-\int_0^T \int_{\omega}v_{\widetilde\beta}\psi-\int_0^T \int_{{\mathbb{T}^{d}_{L}}}f_{\theta_{1},q}\psi=0,
	 \end{array}
	 \end{equation}
	 which coincides with \eqref{weakformulationaftercontrol}.\\
	 In particular, \eqref{EulerLagrange} holds for all smooth functions $\psi$ on $\overline{{\mathbb{T}^{d}_{L}}}\times[0,T]$ with $(\psi,\partial_{t}\psi)(\cdot,T)=0,$ which implies that $\widetilde{g}$ with $v_{\beta}=v_{\widetilde{\beta}}$ solves \eqref{dampedbeamw} in the sense of transposition. Hence comparing \eqref{EulerLagrange} and \eqref{weakformulationaftercontrol} and using the uniqueness of transposition solution we have shown that there exists a control $v_{\widetilde{\beta}}$ which drives the solution of the system \eqref{dampedbeamw} to the null state at time $T.$ \\
	 Now we aim to show that the control function $v_{\beta}=v_{\widetilde{\beta}}\in L^{2}(0,T;L^{2}(\mathbb{T}^{d}_{L})).$ In that direction we first observe that
	 $$J(\psi_{\min})\leqslant J(0)=0.$$
	 This gives
	 \begin{equation}\label{usingminimization}
	 \begin{array}{ll}
	 &\displaystyle\frac{1}{2}\|\psi\|^{2}_{obs}=\frac{1}{2}\int_0^T \int_{{\mathbb{T}^{d}_{L}}}|(\partial_{tt}+\partial_{txx}+\partial_{xxxx}+a)\psi_{\min}|^{2}e^{-2s\phi}+\frac{s^{7}\lambda^{8}}{2}\int_0^T \int_{{\mathbb{T}^{d}_{L}}}\chi_{\omega}^{2}\xi^{7}|\psi_{\min}|^{2}e^{-2s\phi}\\
	 &\displaystyle\leqslant \int_0^T \int_{{\mathbb{T}^{d}_{L}}}f_{\theta_{1},q}\psi_{\min}\\
	 &\displaystyle \leqslant C\|\psi_{\min}\|_{obs}\|(\beta_{0},\beta_{1})\|_{H^{3}({\mathbb{T}^{d}_{L}})\times H^{1}({\mathbb{T}^{d}_{L}})}\\
	 & \displaystyle \leqslant C\left(\frac{\epsilon}{2}\|\psi_{\min}\|^{2}_{obs}+\frac{1}{2\epsilon}\|(\beta_{0},\beta_{1})\|^{2}_{H^{3}({\mathbb{T}^{d}_{L}})\times H^{1}({\mathbb{T}^{d}_{L}})}\right).
	 \end{array}
	 \end{equation} 
	 for some positive constant $C$ and a positive parameter $\epsilon.$ The expression \eqref{usingminimization}$_{3}$ from \eqref{usingminimization}$_{2}$ is obtained since the map \eqref{contextension} admits of a continuous extension on $H_{obs}$ defined by \eqref{Holder}.\\
	 Choosing small enough value of the positive parameter $\epsilon,$ one obtains the following from \eqref{usingminimization}
	 \begin{equation}\label{boundcontrol}
	 \begin{array}{l}
	 \displaystyle{s^{7}\lambda^{8}}\int_0^T \int_{{\mathbb{T}^{d}_{L}}}\chi_{\omega}^{2}\xi^{7}|\psi_{\min}|^{2}e^{-2s\phi}\leqslant C\|(\beta_{0},\beta_{1})\|^{2}_{H^{3}({\mathbb{T}^{d}_{L}})\times H^{1}({\mathbb{T}^{d}_{L}})},
	 \end{array}
	 \end{equation}
	 for some positive constant $C.$\\
	 Using the fact that
	 $$\xi^{7}e^{-2s\phi}<C\quad\mbox{on}\quad {\mathbb{T}^{d}_{L}}\times (0,T),$$
	 for some positive constant $C,$ and the estimate \eqref{boundcontrol}, one establishes the following bound on the control function $v_{\widetilde{\beta}},$ defined in \eqref{expressionsolutioncontrol}
	 \begin{equation}\label{finalboundcontrol}
	 \begin{array}{l}
	 \displaystyle\|v_{\widetilde{\beta}}\|_{L^{2}(0,T;L^{2}({\mathbb{T}^{d}_{L}}))}\leqslant C\|(\beta_{0},\beta_{1})\|_{H^{3}({\mathbb{T}^{d}_{L}})\times H^{1}({\mathbb{T}^{d}_{L}})},
	 \end{array}
	 \end{equation}
	 for some positive constant $C.$ This proves our claim.\\
	  %In view of \eqref{finalboundcontrol} and using Lemma \ref{lemmaexistence} one can uniquely solve \eqref{dampedbeamw} in the functional framework
	  %\begin{equation}\label{regularityw}
	  %\begin{array}{l}
	 %\displaystyle(g,\partial_{t}g)\in L^{2}(0,T;H^{4}({\mathbb{T}^{d}_{L}})\times H^{2}({\mathbb{T}^{d}_{L}}))\cap H^{1}(0,T;H^{2}({\mathbb{T}^{d}_{L}})\times L^{2}({\mathbb{T}^{d}_{L}})).
	 % \end{array}
	 % \end{equation}
	 %Since $\widetilde{g}$ is a transposition solution to \eqref{dampedbeamw} evolving from the data $(\beta_{0},\beta_{1})$ and the control $v_{\widetilde{\beta}},$ hence using the uniqueness of transposition solution one concludes that $\widetilde{g}$ is a unique strong solution of the system \eqref{dampedbeamw} and lies in the functional spaces \eqref{regularityw}.\\
	  %Using \eqref{EulerLagrange} once again, we remark that it coincides with \eqref{weakformulationaftercontrol}, hence $\widetilde{g}$ with the control function $v_{\widetilde{\beta}}$ satisfies the controllability requirement \eqref{nullcontrolrequirementw}.\\
	   In view of the decomposition \eqref{decomposebeta} we conclude that the system \eqref{dampedbeam} is null controllable and there exists a control $v_{\beta}\in L^{2}(0,T;L^{2}(\mathbb{T}^{d}_{L}))$ which drives the solution of \eqref{dampedbeam} to the zero state. Finally using the regularity result from Lemma \ref{existencepotential} we conclude that the controlled trajectory $\beta$ satisfies the regularity \eqref{regularitytrajectory}. 
\end{proof}
\appendix
\section{Proof of Lemma \ref{existencepotential}}\label{appendix}
\begin{proof}[Proof of Lemma \ref{existencepotential}]
	The proof of Lemma \ref{existencepotential} will be a consequence of the following result on the analyticity of a damped beam semigroup:
	\begin{lem}\label{lemmaexistence}
		Let $$(\beta_{0},\beta_{1})\in H^{3}({\mathbb{T}^{d}_{L}})\times H^{1}({\mathbb{T}^{d}_{L}})$$ and $f\in L^{2}(0,{\kappa};L^{2}({\mathbb{T}^{d}_{L}})).$
		Then the following system 
		\begin{equation}\label{dampedbeamgeneral}
		\left\{ \begin{array}{ll}
		\displaystyle\partial_{tt}\beta-\partial_{txx}\beta+\partial_{xxxx}\beta=f \,& \mbox{in}\, {\mathbb{T}^{d}_{L}}\times(0,{\kappa}),
		\vspace{1.mm}\\
		\displaystyle\beta(\cdot,0)=\beta_{0}\quad \mbox{and}\quad \partial_{t}\beta(\cdot,0)=\beta_{1}\, &\mbox{in}\, {\mathbb{T}^{d}_{L}},
		\end{array}\right.
		\end{equation}
		admits a unique solution in the following functional framework
		\begin{equation}\label{functionalframework}
		\begin{array}{l}
		\beta\in L^{2}(0,{\kappa};H^{4}({\mathbb{T}^{d}_{L}}))\cap H^{1}(0,{\kappa};H^{2}({\mathbb{T}^{d}_{L}}))\cap H^{2}(0,{\kappa};L^{2}({\mathbb{T}^{d}_{L}})).
		\end{array}
		\end{equation}
		Let us fix a positive constant $\overline{{\kappa}}>{\kappa}>0.$ There exists a positive constant $C=C(\overline{{\kappa}})>0,$ independent of ${\kappa},$ such that the following holds
		\begin{equation}\label{inequalityexistence}
		\begin{array}{ll}
		&\displaystyle\|(\beta,\partial_{t}\beta)\|_{L^{2}(0,{\kappa};H^{4}({\mathbb{T}^{d}_{L}})\times H^{2}({\mathbb{T}^{d}_{L}}))\cap H^{1}(0,{\kappa};H^{2}({\mathbb{T}^{d}_{L}})\times L^{2}({\mathbb{T}^{d}_{L}}))}+\|(\beta,\partial_{t}\beta)\|_{L^{\infty}(0,{\kappa};H^{3}(\mathbb{T}^{d}_{L})\times H^{1}(\mathbb{T}^{d}_{L}))}\\
		&\leqslant\displaystyle C(\|(\beta_{0},\beta_{1})\|_{H^{3}({\mathbb{T}^{d}_{L}})\times H^{1}({\mathbb{T}^{d}_{L}})}
		\displaystyle+\|f\|_{L^{2}(0,{\kappa};L^{2}({\mathbb{T}^{d}_{L}}))}).
		\end{array}
		\end{equation}
	\end{lem}	
	\begin{proof}
		We write \eqref{dampedbeamgeneral} in the following form:
		\begin{equation}\label{matrixform}
		\left\{ \begin{array}{ll}
		\partial_{t}\begin{pmatrix}
		\beta\\
		\partial_{t}\beta
		\end{pmatrix}-\begin{pmatrix}
		0 & I\\
		-\partial_{xxxx} & \partial_{xx}
		\end{pmatrix}\begin{pmatrix}
		\beta\\\partial_{t}\beta
		\end{pmatrix}=\begin{pmatrix}
		\beta\\
		\partial_{t}\beta
		\end{pmatrix} & \quad\mbox{in}\,\,{\mathbb{T}^{d}_{L}}\times(0,{\kappa}),\\
		\beta(\cdot,0)=\beta_{0}\quad\mbox{and}\quad \partial_{t}\beta(\cdot,0)=\beta_{1}&\quad\mbox{in}\,\,{\mathbb{T}^{d}_{L}}.
		\end{array}\right.
		\end{equation}
		Since we are on a one dimensional torus ${\mathbb{T}^{d}_{L}},$ it is easy to see that the operator
		$$\mathcal{A}=\begin{pmatrix}
		0 & I\\
		-\partial_{xxxx} & \partial_{xx}
		\end{pmatrix}$$
		is defined in $H^{2}({\mathbb{T}^{d}_{L}})\times L^{2}({\mathbb{T}^{d}_{L}})$ with the domain
		$$\mathcal{D}(\mathcal{A})=H^{4}({\mathbb{T}^{d}_{L}})\times H^{2}({\mathbb{T}^{d}_{L}}).$$ 
		Further it follows from \cite{chen} that the operator $(\mathcal{A},\mathcal{D}(\mathcal{A}))$ generates an analytic semigroup on $H^{2}({\mathbb{T}^{d}_{L}})\times L^{2}({\mathbb{T}^{d}_{L}}).$\\
		Hence one can apply the isomorphism theorem \cite[Theorem 3.1, p. 143]{ben} to obtain \eqref{functionalframework}. In order to furnish the inequality \eqref{inequalityexistence} with a constant $C=C(\overline{{\kappa}}),$ independent of ${\kappa},$ one can use the technique from the proof of \cite[Theorem 2.7]{Mitraexistence} which involves in extending the non homogeneous term $f$ by zero in a time interval $({\kappa},\overline{{\kappa}})$ and solving \eqref{dampedbeamgeneral} in $(0,\overline{{\kappa}}).$ In particular the bound on $\|(\beta,\partial_{t}\beta)\|_{L^{\infty}(0,{\kappa};H^{3}(\mathbb{T}^{d}_{L})\times H^{1}(\mathbb{T}^{d}_{L}))}$ in \eqref{inequalityexistence} can be obtained by a priori estimate in the spirit of \cite[Eq. 2.36]{Mitraexistence}.
	\end{proof}
	\begin{remark}
		It might seem surprising to include an estimate of $\|(\beta,\partial_{t}\beta)\|_{L^{\infty}(0,{\kappa};H^{3}(\mathbb{T}^{d}_{L})\times H^{1}(\mathbb{T}^{d}_{L}))}$ in \eqref{inequalityexistence}, since it can be obtained by interpolation from the estimate of\\ $\|(\beta,\partial_{t}\beta)\|_{L^{2}(0,{\kappa};H^{4}({\mathbb{T}^{d}_{L}})\times H^{2}({\mathbb{T}^{d}_{L}}))\cap H^{1}(0,{\kappa};H^{2}({\mathbb{T}^{d}_{L}})
			\times L^{2}({\mathbb{T}^{d}_{L}}))}.$ Using interpolation argument might yield a constant depending on $\kappa$ but for our purpose of obtaining a existence result we need a bound on\\ $\|(\beta,\partial_{t}\beta)\|_{L^{\infty}(0,{\kappa};H^{3}(\mathbb{T}^{d}_{L})\times H^{1}(\mathbb{T}^{d}_{L}))}$ where the constant $C$ must be independent of $\kappa.$ This is the reason why we separate the estimate of $\|(\beta,\partial_{t}\beta)\|_{L^{\infty}(0,{\kappa};H^{3}(\mathbb{T}^{d}_{L})\times H^{1}(\mathbb{T}^{d}_{L}))}$ in \eqref{inequalityexistence}. 
	\end{remark}
		Now based on Lemma \ref{lemmaexistence}, we will prove Lemma \ref{existencepotential} in two steps $i.e$ (1) a local in time existence result and (2) an iteration argument.\\
		$Step\,\,1.$ Local in time existence: This step is based on a fixed point argument which is performed in a sufficiently small time interval.\\
		Let $\widehat{\beta}\in L^{\infty}(0,{\kappa};L^{2}(\mathbb{T}^{d}_{L})).$ Let us consider
		\begin{equation}\label{dampedbeamgeneral3}
		\left\{ \begin{array}{ll}
		\displaystyle\partial_{tt}\beta-\partial_{txx}\beta+\partial_{xxxx}\beta=-a\widehat{\beta}+v_{\beta}\chi_{\omega} \,& \mbox{in}\, {\mathbb{T}^{d}_{L}}\times(0,{\kappa}),
		\vspace{1.mm}\\
		\displaystyle\beta(\cdot,0)=\beta_{0}\quad \mbox{and}\quad \partial_{t}\beta(\cdot,0)=\beta_{1}\, &\mbox{in}\, {\mathbb{T}^{d}_{L}}.
		\end{array}\right.
		\end{equation} 
		From Lemma \ref{lemmaexistence} we know that the problem \eqref{dampedbeamgeneral3} admits of a unique solution in the functional framework \eqref{functionalframework} and using \eqref{boundvbeta} the solution satisfies the following bound
		\begin{equation}\label{inequalityexistence3}
		\begin{array}{ll}
		&\displaystyle\|(\beta,\partial_{t}\beta)\|_{L^{2}(0,{\kappa};H^{4}({\mathbb{T}^{d}_{L}})\times H^{2}({\mathbb{T}^{d}_{L}}))\cap H^{1}(0,{\kappa};H^{2}({\mathbb{T}^{d}_{L}})\times L^{2}({\mathbb{T}^{d}_{L}}))}+\|(\beta,\partial_{t}\beta)\|_{L^{\infty}(0,{\kappa};H^{3}(\mathbb{T}^{d}_{L})\times H^{1}(\mathbb{T}^{d}_{L}))}\\
		&\leqslant\displaystyle C(\|(\beta_{0},\beta_{1})\|_{H^{3}({\mathbb{T}^{d}_{L}})\times H^{1}({\mathbb{T}^{d}_{L}})}
		\displaystyle+\|a\|_{L^{\infty}(\mathbb{T}^{d}_{L}\times(0,{\kappa}))}\|\widehat{\beta}\|_{L^{2}(0,{\kappa};L^{2}({\mathbb{T}^{d}_{L}}))}),
		\end{array}
		\end{equation}
		for some positive constant $C$ independent of ${\kappa}.$ Further in view of the inequality
		$$\|\widehat{\beta}\|_{L^{2}(0,{\kappa};L^{2}({\mathbb{T}^{d}_{L}}))}\leqslant {{\kappa}^{1/2}}\|\widehat{\beta}\|_{L^{\infty}(0,{\kappa};L^{2}({\mathbb{T}^{d}_{L}}))},$$
		one furnishes the following 
		\begin{equation}\label{inequalityexistence4}
		\begin{array}{ll}
		&\displaystyle\|(\beta,\partial_{t}\beta)\|_{L^{2}(0,{\kappa};H^{4}({\mathbb{T}^{d}_{L}})\times H^{2}({\mathbb{T}^{d}_{L}}))\cap H^{1}(0,{\kappa};H^{2}({\mathbb{T}^{d}_{L}})\times L^{2}({\mathbb{T}^{d}_{L}}))}+\|(\beta,\partial_{t}\beta)\|_{L^{\infty}(0,{\kappa};H^{3}(\mathbb{T}^{d}_{L})\times H^{1}(\mathbb{T}^{d}_{L}))}\\
		&\leqslant\displaystyle C(\|(\beta_{0},\beta_{1})\|_{H^{3}({\mathbb{T}^{d}_{L}})\times H^{1}({\mathbb{T}^{d}_{L}})}
		\displaystyle+{{\kappa}}^{1/2}\|a\|_{L^{\infty}(\mathbb{T}^{d}_{L}\times(0,{\kappa}))}\|\widehat{\beta}\|_{L^{\infty}(0,{\kappa};L^{2}({\mathbb{T}^{d}_{L}}))}),
		\end{array}
		\end{equation}
		for some positive constant $C$ independent of $\kappa.$ We will solve the system \eqref{dampedbeam} by proving that the map $\widehat{\beta}\longrightarrow\beta$ from $L^{\infty}(0,{\kappa};L^{2}({\mathbb{T}^{d}_{L}}))$ to itself is a contraction for a sufficiently small time ${\kappa}.$ In that direction let us consider $\widehat{\beta}_{i}$ and  $\widehat{\beta}_{j}$ in the space $L^{\infty}(0,{\kappa};L^{2}({\mathbb{T}^{d}_{L}})).$ Let $\beta_{i}$ and $\beta_{j}$ be the solutions of the problem \eqref{dampedbeamgeneral3} corresponding to the potentials $a\widehat{\beta}_{i}$ and $a\widehat{\beta}_{j}$ respectively. Using the linearity it is easy to observe that $(\beta_{i}-\beta_{j})$ solves system \eqref{dampedbeamgeneral3} with the potential term $a(\widehat{\beta}_{i}-\widehat{\beta}_{j})$ and initial condition $$((\beta_{i}-\beta_{j}),\partial_{t}(\beta_{i}-\beta_{j}))(\cdot,0)=(0,0).$$
		Using \eqref{inequalityexistence4} for $(\beta_{i}-\beta_{j})$ one in particular furnishes the following
		\begin{equation}\label{inequalityexistence5}
		\begin{array}{l}
		\displaystyle\|(\beta_{i}-\beta_{j})\|_{L^{\infty}(0,{\kappa};H^{3}(\mathbb{T}^{d}_{L}))}
		\leqslant\displaystyle C{{\kappa}}^{1/2}\|a\|_{L^{\infty}(\mathbb{T}^{d}_{L}\times(0,{\kappa}))}\|(\widehat{\beta}_{i}-\widehat{\beta}_{j})\|_{L^{\infty}(0,{\kappa};L^{2}({\mathbb{T}^{d}_{L}}))},
		\end{array}
		\end{equation}
		for some positive constant $C$ independent of $\kappa.$ In view of \eqref{inequalityexistence5} we can readily conclude that there exists  $$\kappa^{*}<\frac{1}{C^{2}\|a\|^{2}_{L^{\infty}(\mathbb{T}^{d}_{L}\times(0,T))}},$$ 
		where $C$ is the constant appearing in \eqref{inequalityexistence3} and \eqref{inequalityexistence5} (one can observe that $C$ is the same constant in both of these inequalities), such that the map $\widehat{\beta}\longrightarrow\beta$ from $L^{\infty}(0,{\kappa}^{*};L^{2}({\mathbb{T}^{d}_{L}}))$ to itself is a contraction. Hence by Banach fixed point theorem there exists a unique solution $\beta$ of \eqref{dampedbeam} in the time interval $(0,\kappa^{*})$ and further the choice of $\kappa^{*}$ and the inequality  \eqref{inequalityexistence4} together furnish that
		\begin{equation}\label{inequalityexistence6}
		\begin{array}{ll}
		&\displaystyle\|(\beta,\partial_{t}\beta)\|_{L^{2}(0,{\kappa^{*}};H^{4}({\mathbb{T}^{d}_{L}})\times H^{2}({\mathbb{T}^{d}_{L}}))\cap H^{1}(0,{\kappa^{*}};H^{2}({\mathbb{T}^{d}_{L}})\times L^{2}({\mathbb{T}^{d}_{L}}))}+\|(\beta,\partial_{t}\beta)\|_{L^{\infty}(0,{\kappa^{*}};H^{3}(\mathbb{T}^{d}_{L})\times H^{1}(\mathbb{T}^{d}_{L}))}\\
		&\displaystyle	\leqslant\displaystyle C\|(\beta_{0},\beta_{1})\|_{H^{3}({\mathbb{T}^{d}_{L}})\times H^{1}({\mathbb{T}^{d}_{L}})},
		\end{array}
		\end{equation} 
		for some positive constant $C,$ independent of $\kappa^{*}.$\\
		$Step\,\,2$ Iteration: Using interpolation and the regularity of $\beta$ in $(0,\kappa^{*})$ one obtains that $$(\beta,\partial_{t}\beta)(\cdot,\frac{\kappa^{*}}{2})\in H^{3}(\mathbb{T}^{d}_{L})\times H^{1}(\mathbb{T}^{d}_{L}),$$
		and further from \eqref{inequalityexistence6} one obtains the following
		\begin{equation}\label{boundk*}
		\begin{array}{ll}
		\|(\beta,\partial_{t}\beta)(\cdot,\frac{\kappa^{*}}{2})\|_{H^{3}(\mathbb{T}^{d}_{L})\times H^{1}(\mathbb{T}^{d}_{L})}\leqslant C\|(\beta_{0},\beta_{1})\|_{H^{3}({\mathbb{T}^{d}_{L}})\times H^{1}({\mathbb{T}^{d}_{L}})},
		\end{array}
		\end{equation}
		for some positive constant $C$ independent of $\kappa^{*}.$
		Since the constant $C$ in inequality \eqref{inequalityexistence6} does not depend on the final time $\kappa^{*}$ and the local in time existence result proved in $Step\,\,1,$ is independent of the size of the given data $(\beta_{0},\beta_{1}),$ hence we can once again solve \eqref{dampedbeam} in $(\frac{\kappa^{*}}{2},\frac{3\kappa^{*}}{2})$ with datum $(\beta,\partial_{t}\beta)(\cdot,\frac{\kappa^{*}}{2})$ and obtain 
		\begin{equation}\label{inequalityexistence7}
		\begin{array}{ll}
		&\displaystyle\|(\beta,\partial_{t}\beta)\|_{L^{2}(0,{\kappa^{*}};H^{4}({\mathbb{T}^{d}_{L}})\times H^{2}({\mathbb{T}^{d}_{L}}))\cap H^{1}(0,{\kappa^{*}};H^{2}({\mathbb{T}^{d}_{L}})\times L^{2}({\mathbb{T}^{d}_{L}}))}+\|(\beta,\partial_{t}\beta)\|_{L^{\infty}(0,{\kappa^{*}};H^{3}(\mathbb{T}^{d}_{L})\times H^{1}(\mathbb{T}^{d}_{L}))}\\
		&\displaystyle	\leqslant\displaystyle C\|(\beta,\partial_{t}\beta)(\cdot,\frac{\kappa^{*}}{2})\|_{H^{3}(\mathbb{T}^{d}_{L})\times H^{1}(\mathbb{T}^{d}_{L}}\leqslant C\|(\beta_{0},\beta_{1})\|_{H^{3}({\mathbb{T}^{d}_{L}})\times H^{1}({\mathbb{T}^{d}_{L}})},
		\end{array}
		\end{equation} 
		for some positive constant $C$ independent of $\kappa^{*}.$ In the last line of \eqref{inequalityexistence7} we have used \eqref{boundk*}. One can iterate this argument finitely many times to show that the system \eqref{dampedbeam} renders a unique solution in the time interval $(0,\kappa)$ and the inequality \eqref{inequalityexistence2} holds in the time interval $(0,\kappa)$. Since $\kappa$ is arbitrary we can have the existence result in time interval $(0,T).$ Hence we are done with the proof of Lemma \ref{existencepotential}.
	\end{proof}
 \bibliographystyle{plain}
 \bibliography{bibliography}

 \end{document}